\documentclass[reqno,11pt]{amsart}
\usepackage[letterpaper,hmargin=1.25in,vmargin=0.9in]{geometry}
\usepackage[T1]{fontenc}
\usepackage{lmodern}
\usepackage{microtype}
\usepackage{mathtools,amssymb,mathrsfs}
\usepackage{graphicx,booktabs,tikz,placeins}
\usepackage[unicode,hidelinks]{hyperref}
\allowdisplaybreaks[1]
\newcommand{\R}{\mathbb R}
\DeclareMathOperator{\conv}{conv}
\newcommand{\inc}[3]{\{#1,#2\}\rightsquigarrow #3}
\newcommand{\wit}[4]{(#1,#2;#4)\rightsquigarrow #3}
\newcommand{\simplexfactor}{\lambda_{\mathrm{simplex}}^{(2)}}

\theoremstyle{plain}
\newtheorem{theorem}{Theorem}[section]
\newtheorem{proposition}[theorem]{Proposition}
\newtheorem{lemma}[theorem]{Lemma}

\theoremstyle{definition}

\theoremstyle{remark}

\numberwithin{equation}{section}
\title[Non-Separable Homothetic Triangles, Part I]
{Non-Separable Homothetic Triangles, Part I: Constructions and Lower Bounds}
\author{Yanlu Lian}
\address{School of Mathematics, Hangzhou Normal University, Hangzhou 311121, China}
\address{School of Computational and Mathematical Sciences, Cardiff University, Cardiff, UK}
\email{yllian@hznu.edu.cn}
\author{Fei Xue}
\address{School of Mathematical Sciences, Nanjing Normal University, Nanjing 210046, China}
\email{05429@njnu.edu.cn}
\author{Qihui Yuan}
\address{School of Mathematical Sciences, Nanjing Normal University, Nanjing 210046, China}
\email{18220134@njnu.edu.cn}
\thanks{This work is supported by the NSFC grant (12671118), the Natural Science Foundation of Zhejiang
Province grant (LQ24A010008) and the Scientific Research Fund of Zhejiang
Provincial Education Department grant (Y202351675).}
\date{}
\hypersetup{pdftitle={Non-Separable Homothetic Triangles, Part I: Constructions and Lower Bounds},
pdfauthor={Yanlu Lian, Fei Xue, and Qihui Yuan},pdfsubject={Convex geometry},
pdfkeywords={non-separable family, homothetic triangle, covering factor, cyclic construction}}
\begin{document}
\begin{abstract}
A finite family of planar convex bodies is called a non-separable family if no line disjoint from its union has at least one member in each open half-plane. In this paper, we prove that there exist finite non-separable families of positive homothetic triangles with covering factors strictly exceeding the sharp three-member bound $\mu = \frac{2}{3} + \frac{2}{3\sqrt{3}}$. This resolves negatively a question posed by K. Bezdek and Z. Lángi, who originally established this three-member bound after proving that the classic factor 1 covering theorem by A. W. Goodman and R. E. Goodman for disks fails for arbitrary positive homothets. Besides giving explicit algebraic examples with four, five, and six members having factors of approximately 1.0533161, 1.0551900, and 1.0572061 respectively, we provide a common cyclic recurrence that yields a 303-member family with the exact factor 250000000/235141779. Finally, we derive a continuous model suggested by increasingly fine recurrences, yielding a numerical candidate of 1.0633083;
its attainability and optimality remain open.
\end{abstract}
\maketitle
\section{Introduction}\label{sec:introduction}

A finite family of planar convex bodies is \emph{non-separable}, or an
\emph{NS-family}, if no line disjoint from its union has at least one member
in each open half-plane. For a convex body $K$ and a family
$\mathcal F=(\xi_i+\tau_iK)_{i=1}^n$, with $\tau_i>0$, its
\emph{covering factor} is
\begin{equation}\label{eq:covering-factor}
\lambda(\mathcal F)=
\frac{\min\{s>0:\bigcup_i(\xi_i+\tau_iK)\subseteq\eta+sK
                 \text{ for some }\eta\in\R^2\}}
     {\sum_i\tau_i}.
\end{equation}
Both non-separability and this factor are invariant under nonsingular
affine maps. For a fixed triangle $\Delta$, write
\[
\Lambda_n=\sup\{\lambda(\mathcal F):\mathcal F\text{ is an }n
\text{-member NS-family of positive homothets of }\Delta\},
\]
\[
\simplexfactor=\sup_{n\ge2}\Lambda_n.
\]
The choice of $\Delta$ does not affect these constants. By the reduction
of Bezdek and L\'angi \cite[Section~2]{BezdekLangi2016}, also recalled in
\cite[Theorem~1]{BezdekLangi2025}, $\simplexfactor$ is the corresponding
supremum over positive homothets of arbitrary planar convex bodies.

Goodman and Goodman \cite{GoodmanGoodman1945} proved that an NS-family of
disks is covered by a disk whose radius is the sum of the member radii.
Hadwiger \cite{Hadwiger1947} developed the corresponding theory for convex
systems. The extension with factor $1$ to arbitrary positive homothets
fails: Bezdek and L\'angi \cite{BezdekLangi2016} constructed triangular
examples and proved the sharp three-member bound
\begin{equation}\label{eq:mu}
\Lambda_3=\mu:=\frac23+\frac{2}{3\sqrt3}.
\end{equation}
They asked whether this same factor works for every finite planar
NS-family \cite[Section~3]{BezdekLangi2016}; see also
\cite[Section~1]{BezdekLangi2025}. The upper bound of Akopyan,
Balitskiy and Grigorev \cite[Corollary~2.3]{AkopyanBalitskiyGrigorev2018}
gives $\simplexfactor\le3/2$.

We answer the question about $\mu$ negatively and obtain successively
larger covering factors by changing the configuration and increasing the
number of members. The largest exact value exhibited here is the following.

\begin{theorem}\label{thm:main}
There is a cyclically symmetric NS-family $\mathcal G_{303}$ of $303$
positive homothetic triangles with rational parameters such that
\begin{equation}\label{eq:main-bound}
\lambda(\mathcal G_{303})=\frac{250000000}{235141779}>1.0631883498.
\end{equation}
Consequently $\simplexfactor\ge250000000/235141779$.
\end{theorem}

The construction is completely specified by the integer coordinates in
Appendix~\ref{app:large-data}. Its verification reduces to rational
comparisons using the coordinate test of Section~\ref{sec:model} and a
short argument excluding a separating line. Section~\ref{sec:finite-bounds}
gives that verification and explains how the example was found.

For four, five and six members, we give the configurations with the largest
covering factors found in our numerical searches within their respective
contact models. Each is selected by the first-order stationarity equations
for the covering factor subject to the prescribed contacts. We specify the
parameters as isolated solutions of rational polynomial systems and verify
the constructions by exact rational interval calculations.

\begin{theorem}\label{thm:small-examples}
There are NS-families $\mathcal F_4,\mathcal F_5,\mathcal F_6$ of four,
five and six positive homothetic triangles, respectively, with real-algebraic
parameters and
\[
\begin{array}{c|ccc}
n&4&5&6\\ \hline
\lambda(\mathcal F_n)&1.0533161\ldots&1.0551900\ldots&1.0572061\ldots.
\end{array}
\]
They satisfy $\mu<\lambda(\mathcal F_4)<\lambda(\mathcal F_5)
<\lambda(\mathcal F_6)$. Four is the least cardinality for which the factor
$\mu$ fails.
\end{theorem}

Section~\ref{sec:three-configurations} gives the geometric descriptions and
NS proofs; Appendix~\ref{app:exact} supplies the algebraic definitions,
root-isolation bounds and exact verification. The least-cardinality
assertion uses \eqref{eq:mu} and the two-member bound in Section~\ref{sec:model}.
The calculations establish the examples and their covering factors, but do
not establish their optimality. The sharp four-member problem and its
equality cases are reserved for Part~II; this paper is independent of that
classification.

In each example, three members touch the three sides of a smallest
containing triangle. We call these the \emph{supporting members}, and the
remaining, smaller members \emph{bridges}. The bridges meet selected
segments joining other members. If a line avoids the family, two members
on the same side can only have segments in that half-plane; any member
meeting such a segment must then lie on that side as well. The selected
relations propagate this restriction through the family and rule out
separation. The added bridges thus allow the supporting configuration to
change while preserving non-separability.

The six-member example has threefold cyclic symmetry. Motivated by its
contact pattern, the nine-member example of Section~\ref{sec:nine} uses two
bridges in each cyclic chain, with adjusted supporting members, and has
\[
\lambda(\mathcal G_9)=\frac{625}{591}>\lambda(\mathcal F_6).
\]
For longer chains, Section~\ref{sec:chains} gives a common recurrence and
reduces the prescribed final contact to two product equations. Under
explicit inequalities, those equations guarantee that every intermediate
triangle and segment relation is valid. We also prove that arranging a
fixed collection of steps in increasing order minimizes the total scale.
The supporting configuration and the step values can then be optimized
numerically; one resulting rational family proves Theorem~\ref{thm:main}.

Finally, Section~\ref{sec:continuous} derives a continuous model from small
recurrence steps, with numerical candidate $1.0633083$. Its three supports
remain of positive scale while fine bridges follow three contact curves.
Whether suitably optimized finite chains approach this value, and whether
it is optimal, remain open.

\section{Coordinates and segment relations}\label{sec:model}

After an affine change of coordinates, we may work with triangles contained
in the standard simplex
\[
\Pi=\{z\in\R^3:z_1+z_2+z_3=1\},\qquad
\Sigma=\Pi\cap\R_{\ge0}^3.
\]
The map $z\mapsto(z_2,z_3)$ identifies $\Sigma$ with the planar triangle
$\Delta=\conv\{(0,0),(1,0),(0,1)\}$ used in our drawings.

For a vector $x=(x_1,x_2,x_3)\ge0$ with $|x|_1=x_1+x_2+x_3<1$, put
\begin{equation}\label{eq:defect-triangle}
T(x)=\{z\in\Sigma:z\ge x\}=x+t(x)\Sigma,\qquad t(x)=1-|x|_1.
\end{equation}
All vector inequalities are coordinatewise. Every positive homothet
contained in $\Sigma$ has this form, and $x_r$ is the minimum of its
$r$th coordinate. Thus $x$ records the positions of the three sides and uniquely determines
the triangle. We call $x$ its parameter vector. The number $t(x)$ is the homothety
ratio, also called the \emph{scale} of the member. With $e_r$ denoting the
$r$th standard basis vector of $\R^3$, its vertices are
\begin{equation}\label{eq:vertex-map}
V_r(x)=x+t(x)e_r\qquad(r=1,2,3).
\end{equation}
For a family, we write $T_i=T(x^{(i)})$, $t_i=t(x^{(i)})$, and
$x^{(i)}=(x^{(i)}_1,x^{(i)}_2,x^{(i)}_3)$. The superscript labels the
member; the subscript labels a coordinate.

\begin{lemma}[Minimal container]\label{lem:minimal-container}
For a finite family $(T(x^{(i)}))_{i=1}^n$, the smallest containing positive
homothet of $\Sigma$ has ratio
\begin{equation}\label{eq:minimal-container}
1-\sum_{r=1}^3\min_i x^{(i)}_r.
\end{equation}
\end{lemma}
\begin{proof}
A positive homothet in $\Pi$ has the form $b+s\Sigma$ with $s=1-\sum_rb_r>0$,
and consists of the points $z\in\Pi$ with $z\ge b$. It contains the given
family exactly when $b_r\le m_r:=\min_i x^{(i)}_r$. Its ratio is therefore
at least $1-\sum_rm_r$, attained at $b=m$. This ratio is positive since
$\sum_rm_r\le|x^{(i)}|_1<1$.
\end{proof}
For each coordinate, a member attaining its minimum touches the
corresponding side of the container. The three supporting members need not themselves
form an NS-family. When all three minima are zero, the total scale and
covering factor are
\begin{equation}\label{eq:normalized-cost}
S=\sum_i t(x^{(i)}),\qquad \lambda((T(x^{(i)}))_i)=S^{-1}.
\end{equation}

We prove non-separability by finding segments that join two members and
meet a third. For nonempty convex sets $A_i,A_j,A_k$, write
$\inc{i}{j}{k}$ when $A_k\cap\conv(A_i\cup A_j)\ne\varnothing$.
This means that there are points $a\in A_i$, $b\in A_j$ and a number
$\theta\in[0,1]$ for which $\theta a+(1-\theta)b\in A_k$; equivalently,
\begin{equation}\label{eq:segment-incidence}
A_k\cap(\theta A_i+(1-\theta)A_j)\ne\varnothing
\quad\text{for some }\theta\in[0,1].
\end{equation}
When we record a particular value of $\theta$ satisfying this relation,
we write $\wit{i}{j}{k}{\theta}$. Thus $\theta$ is the coefficient of the
point chosen in $A_i$, and $1-\theta$ is the coefficient of the point
chosen in $A_j$. Neither the points nor the weight need be unique.
The notation applies both to exact vertex contacts and to segments that
pass through the interior of $A_k$. We sometimes use the names of the
sets in place of their indices.

To see why such relations are useful, suppose that a line avoids every
member. Color the members red on one side and blue on the other. If
$A_i,A_j$ have the same color, every segment joining them stays in that
open half-plane. The relation $\inc{i}{j}{k}$ then forces $A_k$ to have
the same color. We will choose enough verified relations to force all
members to have one color, contradicting separation.

Formally, a \emph{two-coloring} assigns red or blue to each member. It is
\emph{compatible} with the chosen segment relations if
\begin{equation}\label{eq:color-implication}
c(i)=c(j)\quad\Longrightarrow\quad c(k)=c(i)
\end{equation}
whenever $\inc{i}{j}{k}$ is among those relations. Only the relations
used in the proof need to be listed.

\begin{lemma}[Coloring criterion]\label{lem:color-criterion}
If the only compatible two-colorings of a finite family of nonempty convex
sets are constant, then the family is NS.
\end{lemma}
\begin{proof}
A separating line would give a compatible coloring with at least one
member of each color, by the preceding observation.
\end{proof}

The next test reduces a segment relation to one scalar inequality.
It is convenient both for exact rational examples and for interval
calculations with algebraic parameters.

\begin{lemma}[Coordinate test]\label{lem:positive-part}
For parameter vectors $x,y,z$ and $0\le\theta\le1$, the relation
$T(z)\cap(\theta T(x)+(1-\theta)T(y))\ne\varnothing$ is equivalent to
\begin{equation}\label{eq:incidence-M}
M(x,y,z;\theta):=\sum_{r=1}^3
\max\{\theta x_r+(1-\theta)y_r,z_r\}\le1.
\end{equation}
It is also equivalent to each of
\begin{align}
\sum_{r=1}^3(\theta x_r+(1-\theta)y_r-z_r)_+&\le t(z),
\label{eq:positive-part}\\
\sum_{r\in J}(\theta x_r+(1-\theta)y_r)+\sum_{r\notin J}z_r&\le1
\quad(J\subseteq\{1,2,3\}),\label{eq:subset-incidence}
\end{align}
where $a_+=\max\{a,0\}$.
\end{lemma}
\begin{proof}
Put $w=\theta x+(1-\theta)y$. The set
$\theta T(x)+(1-\theta)T(y)$ equals $T(w)$.
It meets $T(z)$ exactly when a point of $\Sigma$ is coordinatewise at least
both $w$ and $z$. Such a point exists precisely when
$\sum_r\max\{w_r,z_r\}\le1$. Subtracting $|z|_1$ gives
\eqref{eq:positive-part}. The sum of the coordinatewise maxima is also
the maximum over the eight subset expressions in
\eqref{eq:subset-incidence}, proving the last equivalence.
\end{proof}
The function $M$ is convex and piecewise linear in $\theta$. Its formula
can change only when one of the two entries in a maximum becomes equal
to the other. Therefore its minimum on $[0,1]$ occurs at an endpoint or
at one of the at most three values
\begin{equation}\label{eq:incidence-breakpoints}
\theta=\frac{z_r-y_r}{x_r-y_r}\in[0,1],\qquad x_r\ne y_r.
\end{equation}
For rational vectors, these finitely many rational evaluations decide
whether the segment relation holds. If a suitable rational weight is
already given, it suffices to check \eqref{eq:incidence-M} at that weight.

A single member has covering factor $1$. Two compact convex members form
an NS-family exactly when they intersect, by strict separation. If
$p\in F_1\cap F_2$, then $F_1\cup F_2\subset F_1+F_2-p$.
For positive homothets of one convex body, the latter is a homothet of
ratio $\tau_1+\tau_2$. Hence
\begin{equation}\label{eq:two-member-bound}
\Lambda_2\le1<\mu.
\end{equation}

\section{Three basic configurations}\label{sec:three-configurations}

We begin with four-, five-, and six-member families. In each case three
members support a smallest triangular container, and the remaining members
lie in its interior. The interior members enforce segment relations that
rule out a separating line. The six-member example consists of two groups
of three triangles, each obtained by repeatedly applying one affine symmetry
to a single triangle. It provides the first instance of the layered
construction developed below.

The candidates are selected by the contact and first-order equations
specified in Appendix~\ref{app:exact}. That appendix verifies their exact
parameters, geometric inequalities and covering factors.
Here we prove non-separability from the resulting segment relations.
In a relation \(\inc{i}{j}{k}\), some segment joining points of the two
source triangles \(T_i,T_j\) meets the target triangle \(T_k\).
As in Lemma~\ref{lem:color-criterion}, a compatible two-coloring assigns the
target the common color of the two sources whenever they have the same color.

\subsection{Four members}\label{subsec:four}

Let \(\mathcal F_4=(T(x^{(i)}))_{i=1}^4\), where the parameter vectors are defined
in Appendix~\ref{app:four}. They satisfy
\begin{equation}\label{eq:four-rules}
 \inc{1}{2}{3},\qquad \inc{2}{4}{3},\qquad \inc{1}{4}{2},
 \qquad \inc{2}{3}{1},\qquad \inc{2}{3}{4}.
\end{equation}
Every compatible two-coloring is constant. Indeed, if members 2 and 3 have
the same color, the last two rules force this color on members 1 and 4.
Otherwise the first two rules force members 1 and 4 to have the color of
member 3, contrary to the third rule. Lemma~\ref{lem:color-criterion}
therefore proves that \(\mathcal F_4\) is non-separable.

The three coordinate minima are zero. Hence \(\Sigma\) is a smallest
container by Lemma~\ref{lem:minimal-container}, and
\[
 \lambda(\mathcal F_4)=\left(\sum_{i=1}^4t(x^{(i)})\right)^{-1}
 \approx1.05331610414687.
\]
The strict rational bounds are given in \eqref{eq:four-value-isolation}.

\subsection{Five members}\label{subsec:five}

For the parameter vectors of Appendix~\ref{app:five}, put
\(\mathcal F_5=(T(x^{(i)}))_{i=1}^5\). The required segment relations are
\begin{equation}\label{eq:five-rules}
\begin{gathered}
 \inc{1}{4}{2},\quad \inc{4}{5}{3},\quad \inc{1}{5}{2},\quad
 \inc{1}{2}{3},\\
 \inc{3}{4}{1},\quad \inc{1}{2}{4},\quad \inc{2}{3}{5}.
\end{gathered}
\end{equation}
If members 1 and 2 have the same color, the rules with source pairs
\((1,2)\) and \((2,3)\) force a constant coloring. Otherwise the first rule
forces \(c(4)=c(2)\); the fifth forces \(c(3)=c(1)\); and the second forces
\(c(5)=c(1)\). The third rule is then violated. Thus every compatible
coloring is constant, and \(\mathcal F_5\) is non-separable.

For one explicit segment witness, let
\(y=\theta_1x^{(1)}+(1-\theta_1)x^{(4)}\). The defining equations give
\[
 y_1+x^{(2)}_{2}+x^{(2)}_{3}=1,
 \qquad y_1+y_2+x^{(2)}_{3}=1.
\]
Thus \(y_2=x^{(2)}_{2}\), \(y_1=x^{(2)}_{1}+t_2\), and comparison of the third
coordinates in \(\Pi\) yields
\begin{equation}\label{eq:five-worked-witness}
 x^{(2)}+t_2e_1
 =\theta_1(x^{(1)}+t_1e_3)+(1-\theta_1)(x^{(4)}+t_4e_3).
\end{equation}
This places a vertex of \(T_2\) on a segment joining vertices of \(T_1\)
and \(T_4\). The other six segment relations are verified in the same coordinate
form. The coordinate minima vanish, so
\[
 \lambda(\mathcal F_5)=\left(\sum_{i=1}^5t(x^{(i)})\right)^{-1}
 \approx1.0551900248209903;
\]
see \eqref{eq:five-value-isolation} for strict rational bounds.
Figure~\ref{fig:four-configuration} displays both constructions.

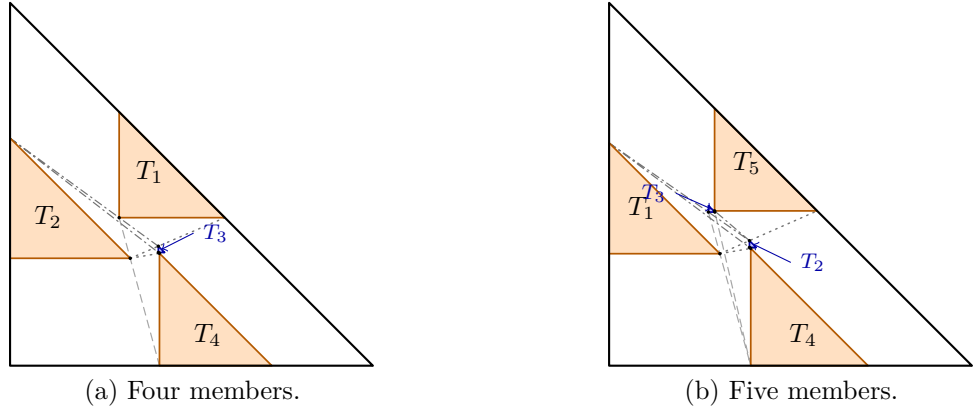
\begin{figure}[htbp]
\centering
\begin{minipage}[b]{.48\textwidth}\centering
\begin{tikzpicture}[x=4.8cm,y=4.8cm,line join=round,line cap=round]
  \coordinate (fpa) at (0.300691,0.407981);
  \coordinate (fqa) at (0.592019,0.407981);
  \coordinate (fra) at (0.300691,0.699309);
  \coordinate (fpb) at (0.000000,0.295998);
  \coordinate (fqb) at (0.331181,0.295998);
  \coordinate (frb) at (0.000000,0.627179);
  \coordinate (fpc) at (0.408671,0.312059);
  \coordinate (fqc) at (0.425878,0.312059);
  \coordinate (frc) at (0.408671,0.329266);
  \coordinate (fpd) at (0.411773,0.000000);
  \coordinate (fqd) at (0.721440,0.000000);
  \coordinate (frd) at (0.411773,0.309667);

  \draw[black!35,densely dashed,line width=.45pt]
    (fpa)--(fpd);
  \draw[black!55,dotted,line width=.55pt]
    (fqa)--(fqb) (fqb)--(fqc);
  \draw[black!55,dash pattern=on 3pt off 1.5pt on .6pt off 1.5pt,
        line width=.45pt]
    (frb)--(frd) (frb)--(frc);

  \filldraw[fill=orange!24,draw=orange!70!black,line width=.65pt]
    (fpa)--(fqa)--(fra)--cycle;
  \filldraw[fill=orange!24,draw=orange!70!black,line width=.65pt]
    (fpb)--(fqb)--(frb)--cycle;
  \filldraw[fill=blue!22,draw=blue!65!black,line width=.65pt]
    (fpc)--(fqc)--(frc)--cycle;
  \filldraw[fill=orange!24,draw=orange!70!black,line width=.65pt]
    (fpd)--(fqd)--(frd)--cycle;

  \fill (frc) circle[radius=.65pt];
  \fill (fpc) circle[radius=.65pt];
  \fill (fqb) circle[radius=.65pt];
  \fill (fpa) circle[radius=.65pt];
  \fill (frd) circle[radius=.65pt];

  \draw[line width=.8pt] (0,0)--(1,0)--(0,1)--cycle;
  \node[font=\small] at (0.385,0.535) {\(T_1\)};
  \node[font=\small] at (0.105,0.410) {\(T_2\)};
  \node[font=\small] at (0.545,0.085) {\(T_4\)};
  \draw[blue!65!black,thin,->]
    (0.505,0.365) node[right,font=\scriptsize] {\(T_3\)}
    -- (0.418,0.321);
\end{tikzpicture}
\par\small (a) Four members.
\end{minipage}\hfill
\begin{minipage}[b]{.48\textwidth}\centering
\begin{tikzpicture}[x=4.8cm,y=4.8cm,line join=round,line cap=round]
  \coordinate (vpa) at (0.000000,0.308668);
  \coordinate (vqa) at (0.305317,0.308668);
  \coordinate (vra) at (0.000000,0.613985);
  \coordinate (vpb) at (0.386876,0.325869);
  \coordinate (vqb) at (0.405562,0.325869);
  \coordinate (vrb) at (0.386876,0.344555);
  \coordinate (vpc) at (0.273710,0.423367);
  \coordinate (vqc) at (0.291481,0.423367);
  \coordinate (vrc) at (0.273710,0.441137);
  \coordinate (vpd) at (0.390375,0.000000);
  \coordinate (vqd) at (0.713638,0.000000);
  \coordinate (vrd) at (0.390375,0.323263);
  \coordinate (vpe) at (0.290721,0.426619);
  \coordinate (vqe) at (0.573381,0.426619);
  \coordinate (vre) at (0.290721,0.709279);

  \draw[black!35,densely dashed,line width=.45pt]
    (vpd)--(vpe) (vpc)--(vpd);
  \draw[black!55,dotted,line width=.55pt]
    (vqa)--(vqe) (vqa)--(vqb);
  \draw[black!55,dash pattern=on 3pt off 1.5pt on .6pt off 1.5pt,
        line width=.45pt]
    (vra)--(vrd) (vra)--(vrb) (vrb)--(vrc);

  \filldraw[fill=orange!24,draw=orange!70!black,line width=.65pt]
    (vpa)--(vqa)--(vra)--cycle;
  \filldraw[fill=blue!22,draw=blue!65!black,line width=.65pt]
    (vpb)--(vqb)--(vrb)--cycle;
  \filldraw[fill=blue!22,draw=blue!65!black,line width=.65pt]
    (vpc)--(vqc)--(vrc)--cycle;
  \filldraw[fill=orange!24,draw=orange!70!black,line width=.65pt]
    (vpd)--(vqd)--(vrd)--cycle;
  \filldraw[fill=orange!24,draw=orange!70!black,line width=.65pt]
    (vpe)--(vqe)--(vre)--cycle;

  \fill (vpb) circle[radius=.65pt];
  \fill (vqc) circle[radius=.65pt];
  \fill (vrb) circle[radius=.65pt];
  \fill (vpc) circle[radius=.65pt];
  \fill (vqa) circle[radius=.65pt];
  \fill (vrd) circle[radius=.65pt];
  \fill (vpe) circle[radius=.65pt];

  \draw[line width=.8pt] (0,0)--(1,0)--(0,1)--cycle;
  \node[font=\small] at (0.090,0.435) {\(T_1\)};
  \node[font=\small] at (0.535,0.085) {\(T_4\)};
  \node[font=\small] at (0.380,0.560) {\(T_5\)};
  \draw[blue!65!black,thin,->]
    (0.500,0.285) node[right,font=\scriptsize] {\(T_2\)}
    -- (0.397,0.335);
  \draw[blue!65!black,thin,->]
    (0.185,0.475) node[left,font=\scriptsize] {\(T_3\)}
    -- (0.282,0.432);
\end{tikzpicture}
\par\small (b) Five members.
\end{minipage}
\caption{The four- and five-member families in the planar image
\(\Delta\) of \(\Sigma\). Orange members support the container; blue
members lie in its interior. The line styles distinguish the three source
vertex types, and dots mark the targets of the segment witnesses.
Coordinates are rounded only for drawing.}
\label{fig:four-configuration}\label{fig:five-configuration}
\end{figure}

\subsection{Six members}\label{subsec:six}

Put \(R(x,y)=(1-x-y,x)\). This affine map has order three and preserves
\(\Delta\). Let \(a,b,X,Y,g\) and \(p_*,v_*,w_*\) be the isolated
algebraic parameters of Appendix~\ref{app:six}. In auxiliary planar
coordinates define
\begin{equation}\label{eq:six-triangles}
 T_1=(1+a,-b)+b\Delta,\qquad T_4=(X,Y)+g\Delta,
\end{equation}
and take their cyclic images
\(T_2=R(T_1)\), \(T_3=R^2(T_1)\),
\(T_5=R(T_4)\), and \(T_6=R^2(T_4)\).

Label the vertices equivariantly, starting with
\[
\begin{aligned}
 p_1&=(1+a,-b),&q_1&=(1+a+b,-b),&r_1&=(1+a,0),\\
 p_4&=(X,Y),&q_4&=(X+g,Y),&r_4&=(X,Y+g).
\end{aligned}
\]
Thus each vertex label is transported by \(R\); it does not denote a fixed
coordinate direction after rotation. The defining equations give
\begin{equation}\label{eq:six-basic-incidences}
 r_5=w_*q_2+(1-w_*)r_1,\qquad
 p_1=(1-p_*)p_6+p_*r_4,\qquad
 r_2=(1-v_*)p_6+v_*r_1.
\end{equation}
All three weights lie in \((0,1)\). Their cyclic images give
\begin{align}\label{eq:six-nine-rules}
 \inc{4}{6}{1},&\quad \inc{4}{5}{2},&\quad \inc{5}{6}{3},\nonumber\\
 \inc{1}{6}{2},&\quad \inc{2}{4}{3},&\quad \inc{3}{5}{1},\\
 \inc{1}{3}{4},&\quad \inc{1}{2}{5},&\quad \inc{2}{3}{6}.\nonumber
\end{align}
The list is invariant under \((1\,2\,3)(4\,5\,6)\). If members 4, 5,
and 6 have one color, the first row forces a constant coloring. Otherwise
rotate the indices so that \(c(4)=c(5)\ne c(6)\). The rules
\(\inc{4}{5}{2}\), \(\inc{2}{4}{3}\), and \(\inc{2}{3}{6}\), in that
order, give a contradiction. Thus \(\mathcal F_6=(T_i)_{i=1}^6\) is
non-separable.

The smallest container is
\begin{equation}\label{eq:six-container}
 \mathcal K_6=(Y,Y)+(1-3Y)\Delta.
\end{equation}
Containment follows from the inequalities in Appendix~\ref{app:six}.
For minimality, the edge \([p_4,q_4]\) and its cyclic images lie on
\(y=Y\), \(x+y=1-Y\), and \(x=Y\). Any containing homothet
\((\alpha,\beta)+s\Delta\) therefore satisfies
\[
 \alpha\le Y,\qquad\beta\le Y,
 \qquad\alpha+\beta+s\ge1-Y,
\]
whence \(s\ge1-3Y\). Consequently
\begin{equation}\label{eq:six-value}
 \lambda(\mathcal F_6)=\lambda_*=\frac{1-3Y}{3(b+g)}
 \approx1.0572061136.
\end{equation}
Its strict rational enclosure is \eqref{eq:six-value-isolation}.

To put the construction in the common normalized container, apply
\[
 \Phi_Y(x,y)=\left(\frac{x-Y}{1-3Y},\frac{y-Y}{1-3Y}\right).
\]
This map commutes with \(R\) and preserves all the segment identities. The
two normalized generators are
\begin{equation}\label{eq:six-normalized-generators}
\begin{aligned}
 \widehat T_1&=\left(\frac{1+a-Y}{1-3Y},\frac{-b-Y}{1-3Y}\right)
                         +\frac b{1-3Y}\Delta,\\
 \widehat T_4&=\left(\frac{X-Y}{1-3Y},0\right)
                         +\frac g{1-3Y}\Delta.
\end{aligned}
\end{equation}
Numerically they are
\[
\begin{aligned}
 \widehat T_1&\approx(0.4011380774,0.2986580018)+0.0193641630\Delta,\\
 \widehat T_4&\approx(0.4048255232,0)+0.2959322859\Delta.
\end{aligned}
\]
The three images of \(\widehat T_4\) support the container, and the three
images of \(\widehat T_1\) lie in its interior. Each triple is an orbit: it
consists of one triangle and its two images under \(R\) and \(R^2\).
We call them the support orbit and the bridge orbit, respectively, and
extend this geometry to cyclic chains.

\begin{figure}[htbp]
\centering
\begin{tikzpicture}[x=8.2cm,y=8.2cm,line join=round,line cap=round]
  \coordinate (p1) at (0.401138,0.298658);
  \coordinate (q1) at (0.420502,0.298658);
  \coordinate (r1) at (0.401138,0.318022);
  \coordinate (p2) at (0.300204,0.401138);
  \coordinate (q2) at (0.280840,0.420502);
  \coordinate (r2) at (0.280840,0.401138);
  \coordinate (p3) at (0.298658,0.300204);
  \coordinate (q3) at (0.298658,0.280840);
  \coordinate (r3) at (0.318022,0.280840);
  \coordinate (p4) at (0.404826,0);
  \coordinate (q4) at (0.700758,0);
  \coordinate (r4) at (0.404826,0.295932);
  \coordinate (p5) at (0.595174,0.404826);
  \coordinate (q5) at (0.299242,0.700758);
  \coordinate (r5) at (0.299242,0.404826);
  \coordinate (p6) at (0,0.595174);
  \coordinate (q6) at (0,0.299242);
  \coordinate (r6) at (0.295932,0.299242);

  \draw[black!35,densely dashed,line width=.45pt]
    (p6)--(r4) (p4)--(r5) (p5)--(r6);
  \draw[black!55,dotted,line width=.55pt]
    (p6)--(r1) (p4)--(r2) (p5)--(r3);
  \draw[black!55,dash pattern=on 3pt off 1.5pt on .6pt off 1.5pt,
        line width=.45pt]
    (q1)--(r3) (q2)--(r1) (q3)--(r2);
  \filldraw[fill=orange!24,draw=orange!70!black,line width=.65pt]
    (p4)--(q4)--(r4)--cycle;
  \filldraw[fill=orange!24,draw=orange!70!black,line width=.65pt]
    (p5)--(q5)--(r5)--cycle;
  \filldraw[fill=orange!24,draw=orange!70!black,line width=.65pt]
    (p6)--(q6)--(r6)--cycle;
  \filldraw[fill=blue!22,draw=blue!65!black,line width=.65pt]
    (p1)--(q1)--(r1)--cycle;
  \filldraw[fill=blue!22,draw=blue!65!black,line width=.65pt]
    (p2)--(q2)--(r2)--cycle;
  \filldraw[fill=blue!22,draw=blue!65!black,line width=.65pt]
    (p3)--(q3)--(r3)--cycle;
  \draw[line width=.8pt] (0,0)--(1,0)--(0,1)--cycle;
  \draw[blue!65!black,thin,->]
    (0.49,0.35) node[right,font=\scriptsize] {\(T_1\)}
    -- (0.407,0.309);
  \draw[blue!65!black,thin,->]
    (0.305,0.505) node[left,font=\scriptsize] {\(T_2\)}
    -- (0.288,0.408);
  \draw[blue!65!black,thin,->]
    (0.37,0.21) node[below,font=\scriptsize] {\(T_3\)}
    -- (0.307,0.289);
  \node[font=\small] at (0.535,0.085) {\(T_4\)};
  \node[font=\small] at (0.385,0.535) {\(T_5\)};
  \node[font=\small] at (0.085,0.385) {\(T_6\)};

  \draw[black!50,densely dashed,line width=.4pt]
    (0.387,0.287) rectangle (0.433,0.333);
  \node[font=\scriptsize] at (0.85,1.085)
    {Detail of \(\inc{4}{6}{1}\)};
  \begin{scope}
    \clip (0.62,0.59) rectangle (1.08,1.05);
    \fill[white] (0.62,0.59) rectangle (1.08,1.05);
    \filldraw[fill=orange!24,draw=orange!70!black,line width=.7pt]
      (0.798255232,-2.28)--(3.757578091,-2.28)
      --(0.798255232,0.679322859)--cycle;
    \filldraw[fill=blue!22,draw=blue!65!black,line width=.7pt]
      (0.761380774,0.706580018)--(0.955022404,0.706580018)
      --(0.761380774,0.900221648)--cycle;
    \draw[black!70,densely dashed,line width=.65pt]
      (-3.25,3.671744768)--(0.798255232,0.679322859);
    \draw[black!70,->,line width=.65pt]
      (0.675,0.770430)--(0.625,0.807388);
  \end{scope}
  \draw[black!60,line width=.5pt]
    (0.62,0.59) rectangle (1.08,1.05);
  \fill[blue!65!black] (0.761380774,0.706580018) circle[radius=1.2pt];
  \fill[orange!70!black] (0.798255232,0.679322859) circle[radius=1.2pt];
  \node[font=\scriptsize,anchor=east] at (0.61,0.84)
    {towards \(\widehat p_6\)};
  \node[font=\scriptsize,anchor=north east] at (0.750,0.696)
    {\(\widehat p_1\)};
  \draw[orange!70!black,line width=.4pt]
    (0.995,0.700)--(0.798255232,0.679322859);
  \node[font=\scriptsize,anchor=north east] at (1.055,0.726)
    {\(\widehat r_4\)};
  \node[font=\scriptsize] at (0.827,0.775) {\(\widehat T_1\)};
  \node[font=\scriptsize] at (0.963,0.63) {\(\widehat T_4\)};
\end{tikzpicture}
\caption{The six-member configuration after its smallest container is
normalized to \(\Delta\).  The orange members \(T_4,T_5,T_6\) support its
three sides, and the blue members \(T_1,T_2,T_3\) lie in its interior.
The three line styles represent the three \(R\)-orbits of segment witnesses
in \eqref{eq:six-nine-rules}.  The inset enlarges the marked window tenfold
and shows \(\widehat p_1=(1-p_*)\widehat p_6+p_*\widehat r_4\), which
witnesses \(\inc{4}{6}{1}\).  The segment continues towards
\(\widehat p_6\), outside the inset.}
\label{fig:six-configuration}
\end{figure}
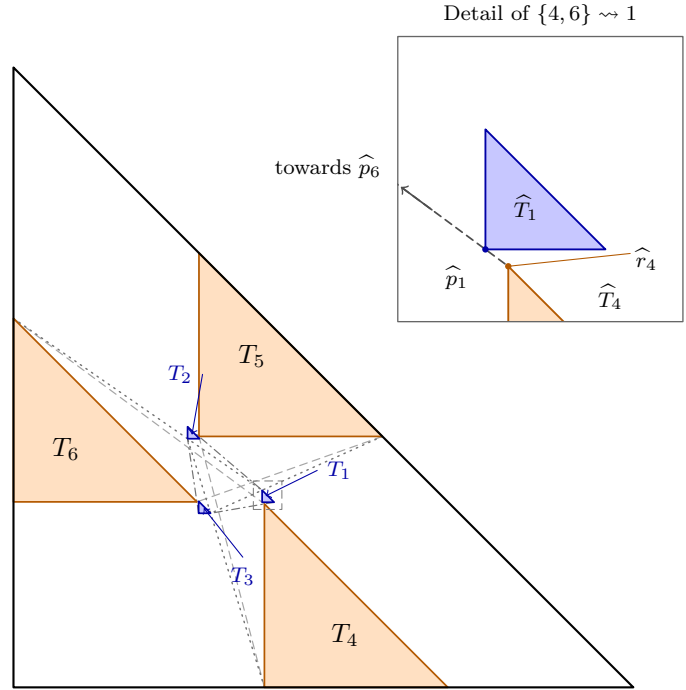

\section{A rational nine-member improvement}\label{sec:nine}

The nine-member construction has three supporting members and two groups
of three bridges. Each group is invariant under the same cyclic symmetry.
The following rational choice permits a short verification of all
required segment relations. Define
\begin{equation}\label{eq:cyclic-map}
\rho(x_1,x_2,x_3)=(x_3,x_1,x_2)
\end{equation}
and take
\begin{equation}\label{eq:nine-vectors}
a=\frac{(2908,4032,2971)}{10000},\quad
b=\frac{(2815,4003,3079)}{10000},\quad
c=\frac{(2992,4048,0)}{10000}.
\end{equation}
Let $A_i=T(\rho^ia)$, $B_i=T(\rho^ib)$ and $C_i=T(\rho^ic)$,
where here and in subsequent cyclic constructions $i$ is read modulo $3$.

\begin{proposition}\label{prop:nine}
The family $\mathcal G_9=(A_i,B_i,C_i)_{i=0}^2$ is NS and has covering factor
\begin{equation}\label{eq:nine-factor}
\lambda(\mathcal G_9)=\frac{625}{591}>\lambda(\mathcal F_6).
\end{equation}
\end{proposition}
\begin{proof}
Each member in the $A$-, $B$- and $C$-groups has scale, respectively, $89/10000$, $103/10000$ and $2960/10000$.
The three supporting members $C_i$ give zero in each coordinate, so the container
is $\Sigma$ and the total scale is $591/625$.
All nine members are distinct: within each cyclic group the coordinates are
different, and the three groups have different member scales.

Table~\ref{tab:nine-incidences} verifies five segment relations using
\eqref{eq:incidence-M}.  Applying $\rho$ once and twice to each relation gives
\begin{equation}\label{eq:nine-rules}
\begin{aligned}
\inc{B_i}{B_{i-1}}{A_i},&\qquad \inc{C_i}{C_{i-1}}{A_i},\\
\inc{A_i}{B_{i-1}}{C_i},&\qquad \inc{A_i}{C_{i-1}}{B_i},\\
&\inc{B_i}{C_{i-1}}{B_{i+1}}.
\end{aligned}
\end{equation}
Suppose first that $B_0,B_1,B_2$ have one color.  The first and third rules
then give that color to every $A_i$ and every $C_i$.
Otherwise, rotate the labels and interchange the colors to make
$B_0,B_1$ red and $B_2$ blue.  The last rule at $i=1,2$ forces $C_0$ blue
and $C_1$ red.  If $C_2$ is blue, the second rule at $i=0$ makes $A_0$
blue, and the fourth makes $B_0$ blue.  If $C_2$ is red, those same rules
at $i=2$ make $A_2$ and $B_2$ red.  Both cases contradict the assigned
colors.  Lemma~\ref{lem:color-criterion} proves NS.

Finally, the upper bound for $\lambda(\mathcal F_6)$ in
\eqref{eq:six-value-isolation} gives the exact comparison
\[
\frac{625}{591}-\frac{105720611360}{10^{11}}
=\frac{119491789}{369375000000}>0.
\]
\end{proof}

\begin{table}[htbp]
\centering
\caption{Exact incidences for the nine-member family.  The listed $\theta$ is the coefficient of a point in the
first member named in each row. Cyclic copies supply all fifteen relations.}
\label{tab:nine-incidences}
\begin{tabular}{ccc}
\toprule Incidence & $\theta$ & $M$\\\midrule
$\inc{B_0}{B_2}{A_0}$&$12/13$&$129983/130000$\\
$\inc{C_0}{C_2}{A_0}$&$211/212$&$529999/530000$\\
$\inc{A_0}{B_2}{C_0}$&$12/13$&$12999/13000$\\
$\inc{A_0}{C_2}{B_0}$&$113/114$&$1$\\
$\inc{B_0}{C_2}{B_1}$&$26/37$&$23124/23125$\\\bottomrule
\end{tabular}
\end{table}

Figure~\ref{fig:nine} shows the family and one segment witnessing each
of the five relations. The marked points may lie in edges or interiors;
three vertices need not be collinear.

\begin{figure}[htbp]
\centering
\input{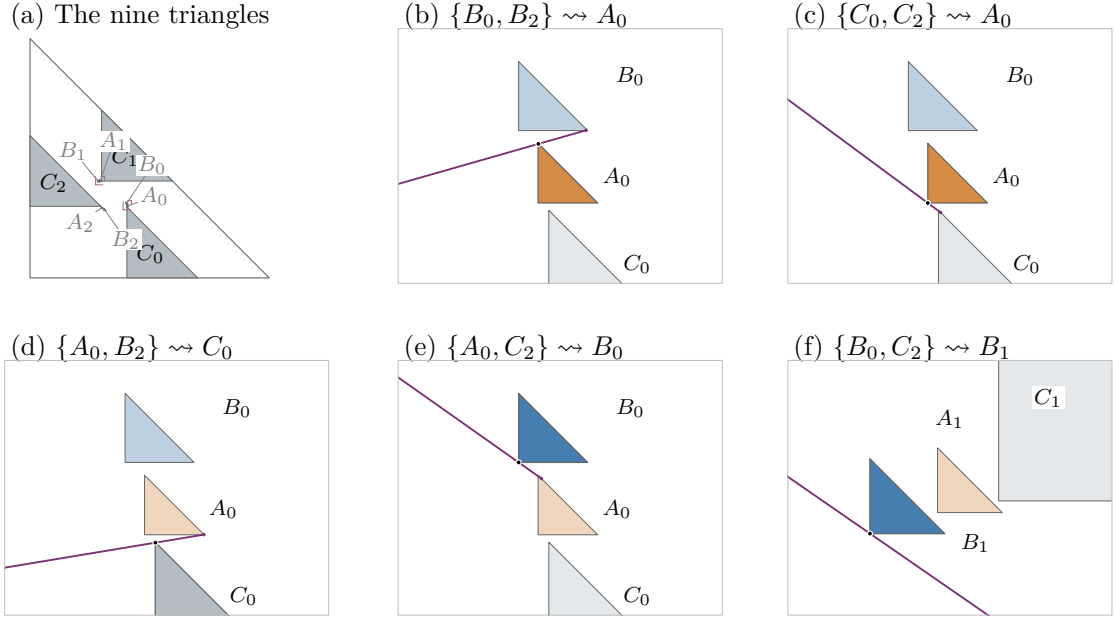}
\caption{The nine-member family and its five basic segment relations.
Supports are gray, $A$-members orange, and $B$-members blue. Panels
(b)--(e) enlarge the lower boxed region in (a), and (f) the upper region.
Purple segments join points in the named source members; black dots lie
in the target. Some source endpoints lie outside the enlarged panels.
The remaining ten relations follow by cyclic permutation.}
\label{fig:nine}
\end{figure}

Thus six bridges arranged in two cyclic groups give a larger covering
factor than the three bridges of $\mathcal F_6$. We next give segment
relations and a recurrence that allow each group to be replaced by a
chain of arbitrary finite length.

\section{Cyclic contact chains}\label{sec:chains}\label{sec:chain}

Keep the three supporting triangles fixed and arrange the bridges in
three cyclic copies of an ordered list. We call this list a \emph{chain};
its order specifies which triangles occur together in the segment
relations below. We first prove non-separability for any list length.
We then choose vertices on the required segments by a recurrence whose
condition on the last triangle reduces to two products of scalar factors.

\subsection{The segment relations}

We first give a sufficient condition for an abstract indexed family of
convex sets. In the construction below, $C_i$ will be the three supports
and $P_{k,i}$ the $k$th bridge in the $i$th cyclic copy. No coordinates
or symmetry are assumed in the lemma.

\begin{lemma}\label{lem:chain-ns}
Let $m\ge1$, let $P_{k,i}$ be nonempty convex sets for
$0\le k\le m$ and $i\in\mathbb Z/3\mathbb Z$, and put $C_i=P_{0,i}$.
The family $(P_{k,i})$ is non-separable if
\begin{align}
 \inc{P_{k-1,i}}{C_{i-1}}{P_{k,i}},\qquad
 \inc{P_{k,i}}{P_{m,i-1}}{P_{k-1,i}}
 &\quad(1\le k\le m),                                  \label{eq:chain-incidences}\\
 \inc{P_{m,i}}{C_{i-1}}{P_{m,i+1}}
 &                                                        \label{eq:chain-terminal-incidence}
\end{align}
hold for every $i\in\mathbb Z/3\mathbb Z$.
\end{lemma}

\begin{proof}
Suppose that a line disjoint from the family separates it, and color each
member according to its open half-plane. Each segment relation forces its target
to have the color of its sources whenever those sources have the same
color. If all three supports have one color, the first relations in
\eqref{eq:chain-incidences}, applied successively in $k$, give that color
to every member.

Otherwise, after cyclic relabeling and interchanging the colors, we may
assume that $C_0,C_1$ are red and $C_2$ is blue. The first relations make
every $P_{k,1}$ red. Equation~\eqref{eq:chain-terminal-incidence} with
$i=1$ then makes $P_{m,2}$ red. Applying the second relations with $i=2$
and $k=m,m-1,\ldots,1$, using the red source $P_{m,1}$ throughout, makes
$C_2$ red. Both cases contradict separation.
\end{proof}

The criterion uses $6m+3$ segment relations, independently of the
coordinates of the members.

\subsection{Choosing vertices on the required segments}

We now give coordinates for the symmetric construction. Put
\[
 \mathbf p^{(0)}=(c,d,0),\qquad
 \mathbf p^{(k)}=(x_k,y_k,z_k),\qquad
 P_{k,i}=T(\rho^i\mathbf p^{(k)}).
\]
Here $x_k,y_k,z_k$ are the three scalar coordinates of the vector
$\mathbf p^{(k)}$; its triangle has ratio $1-x_k-y_k-z_k$.
We prescribe the vertex
$V_3(\mathbf p^{(m)})=(u,v,1-u-v)$ of the last triangle. Equivalently,
its first two coordinates must satisfy $(x_m,y_m)=(u,v)$.
Its ratio will be determined by the last recurrence step. Write
\begin{equation}\label{eq:chain-anchor-scalars}
 \kappa=1-u-v,\qquad D=v-d.
\end{equation}
The vertices used at every step are
\[
 \mathbf q=V_3(\rho^{-1}\mathbf p^{(0)})=(d,0,1-d),\qquad
 \mathbf b=V_2(\rho^{-1}\mathbf p^{(m)})=(v,\kappa,u).
\]
Thus $\mathbf q$ belongs to the support $C_{-1}$, and $\mathbf b$ belongs
to the last triangle $P_{m,-1}$. Although that triangle is still to be
constructed, this vertex is fixed by $u,v$ and does not depend on $z_m$.
We refer to these two fixed vertices as the \emph{anchors}.

For one recurrence step, write $\mathbf p=(x,y,z)$ for the input and
$\mathbf p^+=(x^+,y^+,z^+)$ for the output. The superscript $+$ denotes
the next value of a quantity. In the copy with $i=0$, we require
\begin{equation}\label{eq:chain-vertex-contacts}
 \begin{aligned}
 V_1(\mathbf p^+)&=\theta V_3(\mathbf p)+(1-\theta)\mathbf q,\\
 V_3(\mathbf p)&=\eta V_2(\mathbf p^+)+(1-\eta)\mathbf b,
 \end{aligned}
\end{equation}
The parameters $\theta,\eta\in[0,1]$ specify positions on the two
segments. We call these vertex-on-segment identities \emph{contacts}.
They give the two segment relations in \eqref{eq:chain-incidences}; applying
$\rho$ gives the other copies. We shall obtain weights strictly between
zero and one.

To solve \eqref{eq:chain-vertex-contacts}, introduce the two coordinate
differences that occur in both identities:
\[
 A=1-u-x-y,\qquad B=x+y-d,\qquad \delta=1-\theta.
\]
The first identity gives $y^+=(1-\delta)y$ and
$z^+=1-x-y+\delta B$. The second gives
$x^+=v+(x-v)/\eta$ and $z^+=u+A/\eta$.
Equating the two expressions for $z^+$ yields
$(\eta^{-1}-1)A=\delta B$. Thus, whenever the denominators are nonzero,
\begin{equation}\label{eq:chain-delta-update}
 \begin{aligned}
 x^+&=x+\delta\frac{B(x-v)}{A},&
 y^+&=(1-\delta)y,&
 z^+&=1-x-y+\delta B,\\
 \theta&=1-\delta,&
 \eta&=\frac{A}{A+\delta B}.&&
 \end{aligned}
\end{equation}
Subtracting the sum of the output coordinates from one gives its scale:
\begin{equation}\label{eq:chain-delta-scale}
 t^+=\frac{\delta}{A}\bigl[\kappa(d-x)+Dy\bigr].
\end{equation}
The previous coordinate $z$ does not enter the recurrence because
$V_3(\mathbf p)=(x,y,1-x-y)$.

We next change variables so that successive steps can be combined by
multiplication. This will let us impose the prescribed last vertex without
solving for all intermediate triangles at once. The bracketed expression in
\eqref{eq:chain-delta-scale} is the scalar $E$ below; dividing by
$E$ gives the two new coordinates $R,Y$:
\begin{equation}\label{eq:chain-transform}
 \begin{gathered}
 r=v-x,\qquad L=r-y,\qquad A=\kappa+L,\qquad B=D-L,\\
 E=\kappa r+D(y-\kappa)=\kappa(d-x)+Dy,\qquad
 R=\frac rE,\qquad Y=\frac yE.
 \end{gathered}
\end{equation}
Replace the segment parameter $\delta$ by a step parameter $h$ through
$\delta=hA/(1+hL)$. Substitution in
\eqref{eq:chain-delta-update} gives
\begin{equation}\label{eq:chain-h-update}
 r^+=\frac{r(1+Dh)}{1+hL},\qquad
 y^+=\frac{y(1-\kappa h)}{1+hL},\qquad
 E^+=\frac{E}{1+hL}.
\end{equation}
The last equality follows by inserting the first two into the definition
of $E^+$. The common denominator cancels in $R^+=r^+/E^+$ and $Y^+=y^+/E^+$, giving
\begin{equation}\label{eq:chain-multiplicative}
 R^+=R(1+Dh),\qquad Y^+=Y(1-\kappa h).
\end{equation}
To recover the next triangle from $R^+,Y^+$, use
$\kappa R+DY-1=D\kappa/E$ and the scale formula:
\begin{equation}\label{eq:chain-recovery}
 \begin{gathered}
 E^+=\frac{D\kappa}{\kappa R^++DY^+-1},\qquad
 x^+=v-E^+R^+,\qquad y^+=E^+Y^+,\\
 t^+=hE^+,\qquad z^+=1-x^+-y^+-t^+.
 \end{gathered}
\end{equation}
The next triangle has scale $hE^+$. For the recurrence to give valid
triangles and segment intersections, its denominators must be nonzero,
its parameter vectors nonnegative, its
ratios positive, and its weights in $[0,1]$. We call such a construction
\emph{admissible}. The next lemma derives these requirements from the
initial values and two inequalities at the last step.
Subscripts on the scalar quantities record the step number: $E_k$ is
the value of $E$ after step $k$, and $h_k$ is that step's parameter.

\subsection{Controlling all triangles from the last step}

\begin{lemma}\label{lem:chain-endpoint}
Suppose that $\kappa>0$, $D<0$, $\kappa+D>0$, and $u\ge0$, with
$v=1-u-\kappa$. Start \eqref{eq:chain-h-update} at values satisfying
$r,y,E,A>0$, and take finitely many steps $0<h_k<1/\kappa$.
All updates are defined and preserve these four inequalities. If their
final values satisfy $B_m>0$ and $x_m=v-r_m\ge0$, then $B>0$ and $x\ge0$
at every intermediate step. Under these final conditions, every output in
\eqref{eq:chain-recovery} is a triangle of positive scale with
nonnegative coordinates, and both weights in
\eqref{eq:chain-vertex-contacts} lie strictly between zero and one.
\end{lemma}

\begin{proof}
Because $A>0$, we have $L>-\kappa$ and hence
$1+hL>1-\kappa h>0$. Also $D>-\kappa$ implies $1+Dh>0$.
Besides \eqref{eq:chain-h-update}, direct substitution gives
\begin{equation}\label{eq:chain-ab}
 A^+=\frac{A+h(\kappa+D)r}{1+hL},\qquad
 B^+=\frac{B-h(\kappa+D)y}{1+hL}.
\end{equation}
Thus $r,y,E,A$ stay positive at every step. The second identity shows
that once $B\le0$, all subsequent values of $B$ are negative.
The hypothesis $B_m>0$ therefore forces $B>0$ throughout.
It follows that
\begin{equation}\label{eq:chain-r-monotone}
 r^+-r=\frac{hrB}{1+hL}>0.
\end{equation}
Hence $r_k\le r_m\le v$, proving $x_k\ge0$.

We now have $-\kappa<L<D<0$. The contact weights are
\begin{equation}\label{eq:chain-weights}
 \delta=\frac{hA}{1+hL},\qquad
 \theta=\frac{1-\kappa h}{1+hL},\qquad
 \eta=\frac{1+hL}{1+hD}.
\end{equation}
They satisfy $0<\delta,\theta,\eta<1$ and $\theta=1-\delta$.
Moreover $t^+=hE^+>0$ and
\begin{equation}\label{eq:chain-third-coordinate}
 z^+=1-x-y+\delta B=u+A+\delta B>0.
\end{equation}
The first two output coordinates are nonnegative, so these are the
required triangles and convex combinations.
\end{proof}

We now apply the lemma with support $(c,d,0)$ and prescribed last vertex
$(u,v,\kappa)$. The following strict inequalities ensure the needed
signs for these fixed data:
\begin{equation}\label{eq:chain-domain}
 \begin{gathered}
 0<u<v<d<1-u,\qquad u+v>d,\\
 0<c<v,\qquad c+d<1-u,\qquad
 d(\kappa+D)>\kappa c.
 \end{gathered}
\end{equation}
The first line gives $\kappa>0$, $D<0$, $\kappa+D>0$, and the required
positive final value $B_m=u+v-d$. The second line will give positive
initial values $r_0,A_0,E_0$.

One more segment relation is needed to connect the last triangles in
neighboring cyclic copies, as in \eqref{eq:chain-terminal-incidence}.
We impose
\begin{equation}\label{eq:chain-terminal-vertex}
 V_1(\rho\mathbf p^{(m)})
 =\tau V_3(\mathbf p^{(m)})+(1-\tau)\mathbf q,
 \qquad \tau=\frac uv.
\end{equation}
Its second coordinate gives the displayed value of $\tau$, and its
first coordinate is equivalent to
\begin{equation}\label{eq:chain-conic}
 u^2+uv+v^2-du-(1-d)v=0.
\end{equation}
The third coordinate then follows because both sides have coordinate
sum one. This quadratic equation, a conic in $(u,v)$ for fixed $d$,
describes the allowed last vertices. Choosing $\tau=u/v$ gives its
nonzero branch in rational form:
\begin{equation}\label{eq:chain-conic-parameter}
 v=\frac{1-d+d\tau}{1+\tau+\tau^2},\qquad
 u=\tau v,\qquad
 \kappa=\frac{d+(1-d)\tau^2}{1+\tau+\tau^2},
 \qquad 0<\tau<1.
\end{equation}

To express the requirement on the last vertex in $R,Y$ coordinates,
evaluate \eqref{eq:chain-transform} first at $(x,y)=(c,d)$ and then at
$(x,y)=(u,v)$. The subscripts $0$ and $f$ denote these initial and final
values, respectively:
\begin{equation}\label{eq:chain-endpoint-states}
 \begin{aligned}
 E_0&=d(\kappa+D)-\kappa c,&
 R_0&=\frac{v-c}{E_0},&Y_0&=\frac d{E_0},\\
 E_f&=\kappa(v-u)+D(v-\kappa),&
 R_f&=\frac{v-u}{E_f},&Y_f&=\frac v{E_f}.
 \end{aligned}
\end{equation}
Here $E_0>0$ follows from \eqref{eq:chain-domain}. Substitution of
\eqref{eq:chain-conic-parameter} gives
\[
 E_f=\frac{(1-\tau)(3d^2-3d+1)}{1+\tau+\tau^2}>0,
\]
because $3d^2-3d+1=3(d-1/2)^2+1/4$. Thus $R_0,Y_0,R_f,Y_f$ are
well defined and positive.

By \eqref{eq:chain-multiplicative}, the prescribed last vertex
$V_3(\mathbf p^{(m)})=(u,v,\kappa)$ is equivalent to the two product
equations in the next theorem.

\begin{theorem}\label{thm:chain-construction}
Assume \eqref{eq:chain-domain} and \eqref{eq:chain-conic}.
Let $m\ge1$ and $0<h_k<1/\kappa$ satisfy
\begin{equation}\label{eq:chain-products}
 \prod_{k=1}^m(1+Dh_k)=\frac{R_f}{R_0},\qquad
 \prod_{k=1}^m(1-\kappa h_k)=\frac{Y_f}{Y_0}.
\end{equation}
Starting from $\mathbf p^{(0)}=(c,d,0)$, use
\eqref{eq:chain-h-update} and \eqref{eq:chain-recovery} to construct
$\mathbf p^{(1)},\ldots,\mathbf p^{(m)}$. Then the indexed cyclic family
$(T(\rho^i\mathbf p^{(k)}))_{0\le k\le m,\ i\in\mathbb Z/3\mathbb Z}$
is non-separable, all its scales are positive, and its total scale is
\begin{equation}\label{eq:chain-total-scale}
 S=3\left(1-c-d+\sum_{k=1}^m h_kE_k\right).
\end{equation}
Its coordinate minima are zero in each coordinate. If additionally
$\kappa<v$, its $3(m+1)$ members are pairwise distinct.
\end{theorem}

\begin{proof}
The initial values are
$r_0=v-c>0$, $y_0=d>0$, $A_0=1-u-c-d>0$, and $E_0>0$.
The first part of Lemma~\ref{lem:chain-endpoint} makes every update
well defined. By \eqref{eq:chain-multiplicative} and
\eqref{eq:chain-products}, its final transformed coordinates are $(R_f,Y_f)$.
The inverse formulas therefore give $(x_m,y_m)=(u,v)$. In particular,
$B_m=u+v-d>0$ and $x_m=u>0$, so the rest of the lemma proves the
geometric conditions and the two consecutive contact identities.
The support scale $1-c-d$ is positive as well.

Equation~\eqref{eq:chain-conic} gives
\eqref{eq:chain-terminal-vertex} with $0<\tau<1$.
All segment relations of Lemma~\ref{lem:chain-ns} now hold, proving
non-separability. Formula~\eqref{eq:chain-total-scale} follows from
$t_k=h_kE_k$. All coordinates are nonnegative, and the three
supports attain zero in the three different coordinates.

For distinctness, \eqref{eq:chain-r-monotone} gives
$c=x_0>x_1>\cdots>x_m=u$. Also
$y^+-y=-hAy/(1+hL)<0$, so $y_k\ge v$.
Writing $K=Dr+\kappa y$, we have
\begin{equation}\label{eq:chain-l-monotone}
 L^+-L=\frac{h(K-L^2)}{1+hL},\qquad
 K-L^2=rB+yA>0.
\end{equation}
Thus $x_k+y_k\ge u+v$, and every bridge has
$0<z_k<\kappa<v\le y_k$, while $x_k<c<v$.
Its second coordinate is uniquely largest; the same is true of
the support $(c,d,0)$. Cyclic copies in different sectors consequently
cannot coincide. Within each sector the first coordinates distinguish
the bridges, and their positive third coordinates distinguish them from
the support. Since a triangle determines its parameter vector, all
members are distinct.
\end{proof}

\subsection{Ordering the steps}

Fix the step values and vary only their order. The products in
\eqref{eq:chain-products} are unchanged, so the last vertex remains
$(u,v,\kappa)$. Lemma~\ref{lem:chain-endpoint} also preserves the validity
of every intermediate triangle and segment contact. We can therefore
compare the sums of bridge ratios for the different orders.

\begin{theorem}\label{thm:chain-order}
Suppose that the hypotheses of Lemma~\ref{lem:chain-endpoint}, including
its conditions at the last step, hold for one ordering of fixed step
values, with repetitions allowed. Every permutation is then admissible,
and the nondecreasing arrangement minimizes the sum of bridge ratios
$\sum_{k=1}^m h_kE_k$. Every order that is not nondecreasing has a
strictly larger sum.
In the nondecreasing arrangement the bridge scales are strictly
increasing. In particular, for a family from
Theorem~\ref{thm:chain-construction}, this arrangement minimizes its
total scale among all permutations.
\end{theorem}

\begin{proof}
For every permutation, the first part of
Lemma~\ref{lem:chain-endpoint} makes all updates well defined.
Equation~\eqref{eq:chain-multiplicative} gives the same final $(R,Y)$,
and the inverse formulas give the same final pair $(r,y)$. The remaining part
of the lemma proves admissibility for every permutation.
Consider consecutive steps $a,b$ with the same input coordinates $(x,y)$,
and put $L=r-y$ and $K=Dr+\kappa y$. Their intermediate and final
values of $E$ are
\[
 E_a=\frac E{1+aL},\qquad E_b=\frac E{1+bL},\qquad
 E_{ab}=E_{ba}=\frac E{1+(a+b)L+abK}.
\]
The final identity follows by applying \eqref{eq:chain-h-update} twice.
The two sums of ratios are $aE_a+bE_{ab}$ and $bE_b+aE_{ab}$; subtraction gives
\begin{equation}\label{eq:chain-swap-cost}
 \begin{split}
 &(aE_a+bE_{ab})-(bE_b+aE_{ab})\\
 &\qquad=
 \frac{Eab(a-b)(K-L^2)}
 {(1+aL)(1+bL)\,[1+(a+b)L+abK]}.
 \end{split}
\end{equation}
All denominator factors are positive, and
\eqref{eq:chain-l-monotone} gives $K-L^2>0$.
Exchanging neighboring steps with $a>b$ therefore decreases their total ratio.
The vertex $V_3$ after the pair is unchanged, so all subsequent ratios
are unchanged. Repeated exchanges prove the minimizing assertion.
Finally, $L<0$ implies $E_k>E_{k-1}$; together with nondecreasing positive
$h_k$, this gives strict increase of $h_kE_k$.
\end{proof}

A permutation preserves the last vertex $V_3$ but may change the last
triangle's ratio $h_mE_f$ and its third coordinate. The fixed vertices
$\mathbf q,\mathbf b$ and the final segment relation use only its first
two coordinates, so this change is compatible with the construction.
Theorem~\ref{thm:chain-order} compares orders of fixed step values;
varying those values or their number is a separate problem.

\section{Optimized chains and an exact example}\label{sec:finite-bounds}

The recurrence leaves the support parameters, the terminal contact point
and the step values available for optimization. For a fixed number $m$ of
bridges in each copy, we minimize
\[
3\left(1-c-d+\sum_{k=1}^m h_kE_k\right)
\]
over $c,d,\tau,h_1,\ldots,h_m$, subject to the terminal parametrization
\eqref{eq:chain-conic-parameter}, the two product equations
\eqref{eq:chain-products} and the inequalities of
Theorem~\ref{thm:chain-construction}. Theorem~\ref{thm:chain-order} allows
the steps to be arranged in nondecreasing order.

Table~\ref{tab:finite-bounds} records a branch found by local constrained
numerical minimization. Starting from a shorter chain, we increase the
number of steps and optimize again. The searches checked the intermediate
coordinates, scales and segment weights, with endpoint residuals below
$1.4\cdot10^{-13}$. They did not exhaust the parameter space. The values
in this table are numerical approximations; the exact lower bound below
comes from a separate rational construction.

\begin{table}[htbp]
\centering
\caption{Numerically optimized finite chains. Here $n=3(m+1)$, and
$t_{\max}$ is the largest bridge scale, excluding the three supports.
The displayed values are rounded numerical results.}
\label{tab:finite-bounds}
\begin{tabular}{rrrr}
\toprule $n$ & $S$ & $1/S$ & $t_{\max}$\\\midrule
18&0.9422606658&1.0612774536&0.0082007879\\
27&0.9416566004&1.0619582548&0.0057232823\\
39&0.9412866516&1.0623756305&0.0040821183\\
63&0.9409709114&1.0627321077&0.0025955738\\
153&0.9406704489&1.0630715584&0.0010977097\\
303&0.9405666638&1.0631888611&0.0005594966\\\bottomrule
\end{tabular}
\end{table}

To obtain an exact example, we round the coordinates of the last numerical
candidate slightly downwards. Since $T(x)=\{z\in\Sigma:z\ge x\}$,
this operation enlarges the triangles. The resulting rational vectors
are completely specified in Appendix~\ref{app:large-data}. They need not
satisfy the original vertex equalities; we verify their segment relations
directly.

\begin{proof}[Proof of Theorem~\ref{thm:main}]
Let $\mathbf p^{(k)}=10^{-9}(X_k,Y_k,Z_k)$, $0\le k\le100$, be the
vectors in Appendix~\ref{app:large-data}, and put
\[
P_{k,i}=T(\rho^i\mathbf p^{(k)}),\qquad C_i=P_{0,i}
\quad (i\in\mathbb Z/3\mathbb Z).
\]
Every listed coordinate is nonnegative, and every coordinate sum is less
than one. The listed $X_k$ decrease strictly and every $Y_k$ exceeds
both $X_k$ and $Z_k$, so the $303$ cyclic vectors are distinct. The support
vector is
\[
\mathbf p^{(0)}=10^{-9}(310226993,415666856,0).
\]
It and its cyclic images make the three coordinate minima zero. Thus the
smallest containing homothet is $\Sigma$.

For each of the $201$ relations in \eqref{eq:chain-incidences} and
\eqref{eq:chain-terminal-incidence} with $i=0$, evaluate
\eqref{eq:incidence-M} at $0,1$ and at the values in
\eqref{eq:incidence-breakpoints} that belong to $[0,1]$. Every comparison
is rational, and in every case the least value is less than
$1-1/(4\cdot10^9)$. This verifies all $603$ relations by cyclic symmetry.
Lemma~\ref{lem:chain-ns} therefore proves non-separability. This finite
check requires only the displayed coordinates and the scalar test; no
separately supplied segment weights are needed.

Finally, direct summation gives
\[
\sum_{k=0}^{100}\left(1-10^{-9}(X_k+Y_k+Z_k)\right)
=\frac{78380593}{250000000}.
\]
Multiplying by three gives total scale $235141779/250000000$, and
\eqref{eq:normalized-cost} gives the asserted covering factor.
\end{proof}

\begin{figure}[!t]
\centering
\input{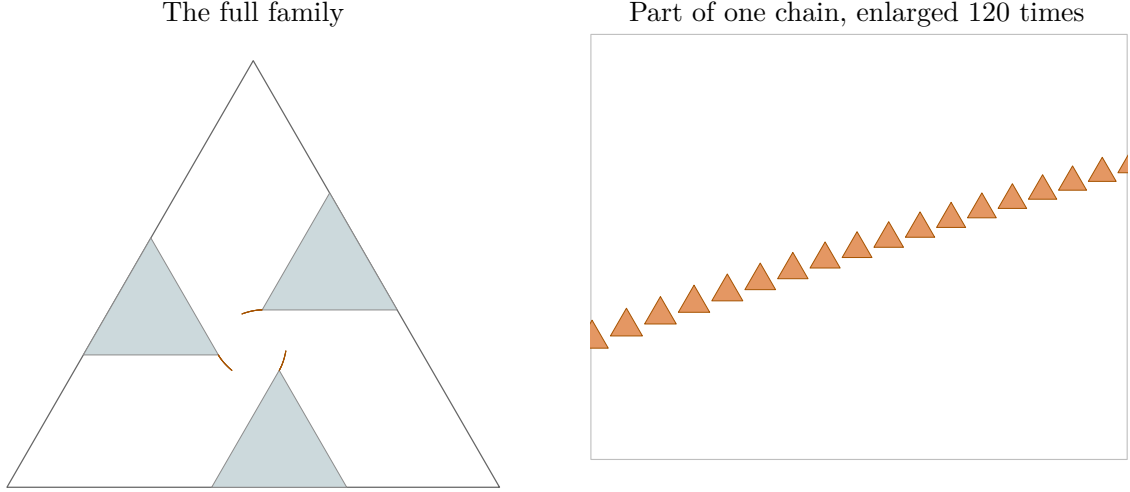}
\caption{The exact $303$-member family. The three large gray triangles
touch the sides of $\Sigma$. Each orange chain contains $100$ bridges.
The right panel enlarges part of one chain $120$ times to make the
individual triangles visible. All drawing coordinates come from the
rational data in Appendix~\ref{app:large-data}.}
\label{fig:large-chain}
\end{figure}

The decreasing largest bridge scales in Table~\ref{tab:finite-bounds}
suggest studying a continuous model for increasingly fine chains.
The table does not prove convergence or improvement for every number
of members; the next section describes the candidate suggested by these
computations.

\section{A continuous model and a candidate value}\label{sec:continuous}

The finite chains suggest a model in which all bridge scales tend to zero.
We describe this model and a candidate value; convergence of finite
NS-families to the proposed optimum remains unproved.  Retain the parameters and
transformed coordinates of Section~\ref{sec:chains}.
For this limiting calculation, keep the support and terminal parameters
fixed and iterate the recurrence without requiring the finite endpoint
equations. We impose the terminal conditions
\eqref{eq:continuous-endpoint} on the limiting curve. Realizing that curve
by exact finite NS chains would require varying the parameters, as
explained below.
If $\max_k h_k\to0$ and $\sum_kh_k\to s$, the two products in
\eqref{eq:chain-products} tend to $e^{Ds}$ and $e^{-\kappa s}$.
The corresponding continuous coordinates are
\begin{equation}\label{eq:continuous-flow}
\begin{gathered}
 R(t)=R_0e^{Dt},\qquad Y(t)=Y_0e^{-\kappa t},\qquad 0\le t\le s,\\
 E(t)=\frac{D\kappa}{\kappa R(t)+DY(t)-1},\qquad
 x(t)=v-E(t)R(t),\qquad y(t)=E(t)Y(t).
\end{gathered}
\end{equation}
The point $(x(t),y(t),1-x(t)-y(t))$ traces the limiting curve of the
vertices $V_3$ used in the contacts.  A finite step contributes scale
$h_kE_k$; when the recovery denominator stays away from zero, these sums
converge to the integral of $E$.  Thus the proposed total scale is
\begin{equation}\label{eq:continuous-cost}
 S_{\mathrm{cont}}=3\left(1-c-d+\int_0^sE(t)\,dt\right).
\end{equation}
The three supports retain scale $1-c-d$, while the integral accounts for
the small bridges in each of the three chains.  Points on the limiting
curve are not themselves positive triangles.

The parameters to optimize are $c,d,\tau,s$, with $s>0$ and $u,v,\kappa$ given by
\eqref{eq:chain-conic-parameter}.  The terminal conditions are
\begin{equation}\label{eq:continuous-endpoint}
 R_0e^{Ds}=R_f,\qquad Y_0e^{-\kappa s}=Y_f.
\end{equation}
We impose \eqref{eq:chain-domain}, allowing $c+d=1-u$, and require the
denominator in \eqref{eq:continuous-flow} to remain nonzero, with
$E,x,y,1-x-y>0$, $A\ge0$ and $B>0$ throughout $[0,s]$.
Here $A=1-u-x-y$ and $B=x+y-d$, as in the finite recurrence.
These conditions define the continuous optimization problem independently
of any finite approximation.

A numerical candidate is given by
\begin{equation}\label{eq:continuous-candidate-parameters}
\begin{aligned}
 c&=0.3104341279739202,&d&=0.41586782506519804,\\
 \tau&=0.6737630444333738,&s&=1.1213014301809736.
\end{aligned}
\end{equation}
At these displayed decimals, numerical evaluation gives endpoint errors
less than $5\cdot10^{-16}$ in each of $x,y$, and
\begin{equation}\label{eq:continuous-candidate-value}
 S_{\mathrm{cont}}\approx0.9404609910968,\qquad
 \lambda_{\mathrm{cand}}:=S_{\mathrm{cont}}^{-1}
       \approx1.0633083238.
\end{equation}
The support scale is approximately $0.2736980469609$, and the integral
in \eqref{eq:continuous-cost} is approximately $0.03978895007138$.
The displayed parameters satisfy
$A_0=1-u-c-d\approx4.18\cdot10^{-8}$, suggesting an optimum on the
boundary $c+d=1-u$.

The search used three local optimization starts within
$0.2\le c\le0.4$, $0.35\le d\le0.5$, $0.5\le\tau\le0.9$ and
$0\le s\le3$.  Two runs also allowed a nonnegative final finite step,
whose parameter approached zero; the third used only the continuous model.
The objective was evaluated by $128$-node Gaussian quadrature, and the
trajectory inequalities were tested at the endpoints and quadrature
nodes.  These runs support the candidate value but do not establish
pathwise feasibility or optimality.  In particular,
\eqref{eq:continuous-candidate-value} is not a proved lower bound for
$\simplexfactor$; the finite bound \eqref{eq:main-bound} remains the
rigorous quantitative conclusion.

Approximating a continuous solution also requires the finite parameters
to change.  Indeed, put
$\Phi(h)=\kappa\log(1+Dh)+D\log(1-\kappa h)$.
Under the stated signs, $\Phi(0)=0$ and $\Phi(h)>0$ for
$0<h<1/\kappa$.  Consequently a finite chain with positive steps has
\[
 \kappa\log(R_f/R_0)+D\log(Y_f/Y_0)=\sum_k\Phi(h_k)>0,
\]
whereas the continuous endpoint equations make the left side zero.
Thus exact finite chains cannot keep both the supports and terminal point
fixed at a continuous solution.  Their parameters would have to vary
with the number of steps, and the suggested boundary $A_0=0$ requires
additional care.

Do suitably optimized finite chains approach this continuous model and
its candidate value?  Does the reciprocal of the infimum in this
continuous problem equal $\simplexfactor$, or can a different contact
pattern give a larger value?
We leave both questions open.  We do not claim an actual countable
family of positive triangles realizing the continuous value.

\appendix
\section{Algebraic data and exact verification}\label{app:exact}

We define the candidates from Section~\ref{sec:three-configurations} by
their contact and first-order stationarity equations, and isolate the
resulting roots in rational boxes. Terminating decimals in the boxes
denote rational numbers. All defining equations and rational centers are
given below. The remaining arithmetic consists of evaluating these equations and their
Jacobians on the specified boxes.

\begin{proposition}[Exact verification]\label{prop:exact-verification}
For each \(n\in\{4,5,6\}\), the system below has exactly one zero in
\(\mathcal B_n\), and its coordinates are real algebraic. The recovered
triangles have positive ratios, all prescribed weights lie in \((0,1)\),
and the segment relations in \eqref{eq:four-rules}, \eqref{eq:five-rules}, and
\eqref{eq:six-nine-rules} hold, respectively. The smallest containers are
\(\Sigma\), \(\Sigma\), and \(\mathcal K_6\), and the covering factors
satisfy \eqref{eq:four-value-isolation}, \eqref{eq:five-value-isolation},
and \eqref{eq:six-value-isolation}.
\end{proposition}

To replace a numerical candidate by an exact point, we prove that a small
rational box contains precisely one solution. Let \(G\) be a rational
polynomial map with \(m\) variables and \(m\) components, write \(J_G\)
for its Jacobian matrix, let
\(\mathcal B=c+[-r,r]^m\), and let \(C\) be a
nonsingular rational matrix. If \(N(z)=z-CG(z)\) maps \(\mathcal B\) into
itself and
\(q:=\sup_{\mathcal B}\|I-CJ_G\|_\infty<1\), the contraction theorem
gives exactly one zero of \(G\) in \(\mathcal B\). Its Jacobian is
nonsingular, and the isolated zero has real-algebraic coordinates.
Here \(\|\cdot\|_\infty\) is the maximum norm and its induced row-sum
norm. Exact rational interval evaluation over the entire boxes gives
\begin{equation}\label{eq:exact-verification-summary}
\begin{array}{c|c|c|c|c}
n&m&r&q&\text{self-map bound}\\ \hline
4&6&10^{-18}&<2.20\cdot10^{-13}&N(\mathcal B_4)\subset\operatorname{int}\mathcal B_4\\
5&33&10^{-18}&<3.299\cdot10^{-16}&\|N(c_5)-c_5\|_\infty/r+q<0.491701\\
6&3&10^{-12}&<7\cdot10^{-7}&\|N(c_6)-c_6\|_\infty/r+q<5.01\cdot10^{-4}
\end{array}
\end{equation}
The last two bounds imply a strict self-map by the mean-value inequality.
The following subsections define the systems and complete the geometric
verification. These are rational interval bounds, not tests at sampled
points; see \cite{Rump2010} for the general arithmetic framework.

\subsection{Four members}\label{app:four}

Represent the three supporting contacts by zero coordinates and write
\[
 x^{(1)}=(0,A,B),\quad x^{(2)}=(C,0,D),\quad
 x^{(3)}=(E,F,G),\quad x^{(4)}=(H,I,0).
\]
We realize each segment relation by placing a target vertex on a segment between
source vertices. In the subset test of Lemma~\ref{lem:positive-part}, this
contact is expressed by making two inequalities equalities. The table lists
their subsets \(J,J'\) for each ordered witness:
\[
\begin{array}{c|cc}
\text{witness}&J&J'\\ \hline
\wit{1}{2}{3}{\theta_1}&\{3\}&\{1,3\}\\
\wit{2}{4}{3}{\theta_2}&\{1\}&\{1,2\}\\
\wit{1}{4}{2}{\theta_3}&\{2\}&\{2,3\}\\
\wit{2}{3}{1}{\theta_4}&\{1\}&\{1,2\}\\
\wit{2}{3}{4}{\theta_5}&\{3\}&\{1,3\}
\end{array}
\]
Subtracting the two corresponding subset equalities gives the coordinate
contacts. Retain the first nine equalities, leaving the last \(J'\) for a
separate equation:
\begin{equation}\label{eq:four-linear-system}
\begin{aligned}
 E&=(1-\theta_1)C,&\theta_1B+(1-\theta_1)D&=1-E-F,\\
 F&=(1-\theta_2)I,&\theta_2C+(1-\theta_2)H&=1-F-G,\\
 D&=\theta_3B,&\theta_3A+(1-\theta_3)I&=1-C-D,\\
 A&=(1-\theta_4)F,&\theta_4C+(1-\theta_4)E&=1-A-B,\\
 &&\theta_5D+(1-\theta_5)G&=1-H-I.
\end{aligned}
\end{equation}
Put \(\vartheta=(\theta_1,\ldots,\theta_5)\),
\(u=(A,B,C,D,E,F,G,H,I)^{\mathsf T}\), and
\(b_4=(0,1,0,1,0,1,0,1,1)^{\mathsf T}\). Move all variable terms to the
left in \eqref{eq:four-linear-system}, reading rows from top to bottom and
within each row from left to right. This defines
\begin{equation}\label{eq:four-matrix-system}
 L_4(\vartheta)u=b_4,\qquad \Delta_4=\det L_4.
\end{equation}
For \(\Delta_4\ne0\), recover the coordinates by
\begin{equation}\label{eq:four-coordinate-chart}
 u(\vartheta)=\operatorname{adj}(L_4)b_4/\Delta_4.
\end{equation}
Set \(\ell_4=(0,0,\theta_5,0,1-\theta_5,0,0,-1,0)\). The remaining
contact and the total scale are determined by the integer polynomials
\begin{equation}\label{eq:four-P4-definition}
\begin{aligned}
 R_4&=\ell_4\operatorname{adj}(L_4)b_4
     =-\det\begin{pmatrix}L_4&b_4\\\ell_4&0\end{pmatrix},\\
 P_4&=\frac{R_4}{\theta_1(\theta_3-1)},\qquad
 N_4=4\Delta_4-\mathbf1^{\mathsf T}\operatorname{adj}(L_4)b_4.
\end{aligned}
\end{equation}
Here \(\mathbf1\) is the all-one column vector. The division defining
\(P_4\) has zero polynomial remainder. In particular,
\begin{equation}\label{eq:four-contact-polynomial}
 \theta_5C+(1-\theta_5)E-H
 =\frac{\theta_1(\theta_3-1)}{\Delta_4}P_4.
\end{equation}
Moreover,
\begin{equation}\label{eq:four-scale-function}
 S_4=4-\mathbf1^{\mathsf T}u=N_4/\Delta_4.
\end{equation}
The search minimizes \(S_4\) while preserving the remaining contact
\(P_4=0\). Its first-order equations require the derivative of \(S_4\) to
vanish along this contact surface. For \(1\le j\le5\), let
\(\partial_j=\partial/\partial\theta_j\) and define
\(\Psi_j=(\partial_jN_4)\Delta_4-N_4\partial_j\Delta_4\).
Clearing denominators in the multiplier equations gives the square system
\(\mathcal G_4\):
\begin{equation}\label{eq:four-KKT-system}
 P_4=0,\qquad \Psi_j-\eta\,\partial_jP_4=0\quad(1\le j\le5),
\end{equation}
in the order \((\vartheta,\eta)\), where \(\eta\) is the multiplier after
the common denominator is cleared. Its defining box is
\begin{equation}\label{eq:four-root-box}
\begin{gathered}
 c_4=(0.297079749755566340,\ 0.007534140977213632,\\
 \phantom{c_4=(}0.725519678908144167,\ 0.264221252988580770,\\
 \phantom{c_4=(}0.148948723710503536,\ 0.000206823449878328),\\
 \mathcal B_4=c_4+[-10^{-18},10^{-18}]^6.
\end{gathered}
\end{equation}

Set \(M_4=J_{\mathcal G_4}(c_4)^{-1}\), which exists over
\(\mathbb Q\), and define the correction
\(N_4^{\mathrm{corr}}(z)=z-M_4\mathcal G_4(z)\). The first row of
\eqref{eq:exact-verification-summary} isolates the root. On its box,
\(\Delta_4>0\), every \(\theta_j\) lies in \((0,1)\), and every nonzero
coordinate and scale exceeds \(0.017\). Among the forty subset inequalities,
ten hold with equality by the defining equations; for each of the remaining
thirty, the right side minus the left side exceeds
\(0.017\). Thus all five segment relations hold. The three prescribed zero
coordinates give the smallest container \(\Sigma\), and
\begin{equation}\label{eq:four-value-isolation}
 1.053316104146871<\lambda(\mathcal F_4)=1/S_4<1.053316104146873.
\end{equation}

\subsection{Five members}\label{app:five}

The supporting members \(T_1,T_4,T_5\) touch the three sides of the
container. In our coordinates these contacts are
\(x^{(1)}_{2}=x^{(4)}_{3}=x^{(5)}_{1}=0\). Use the free-coordinate order
\begin{equation}\label{eq:five-free-coordinate-order}
\begin{split}
 z=(&x^{(1)}_{1},x^{(1)}_{3},x^{(2)}_{1},x^{(2)}_{2},x^{(2)}_{3},x^{(3)}_{1},\\
    &x^{(3)}_{2},x^{(3)}_{3},x^{(4)}_{1},x^{(4)}_{2},x^{(5)}_{2},x^{(5)}_{3}).
\end{split}
\end{equation}
Let \(\theta=(\theta_1,\ldots,\theta_7)\). As in the four-member
construction, each vertex contact is imposed through two subset equalities.
The following row \(q\) selects the subsets \(\{r_q\}\) and
\(\{r_q,s_q\}\) for its witness:
\begin{equation}\label{eq:five-incidence-data}
\begin{array}{c|cc}
\text{witness}&r_q&s_q\\ \hline
\wit{1}{4}{2}{\theta_1}&1&2\\
\wit{4}{5}{3}{\theta_2}&2&3\\
\wit{1}{5}{2}{\theta_3}&3&1\\
\wit{1}{2}{3}{\theta_4}&1&2\\
\wit{3}{4}{1}{\theta_5}&2&3\\
\wit{1}{2}{4}{\theta_6}&3&1\\
\wit{2}{3}{5}{\theta_7}&1&2
\end{array}
\end{equation}
For row \(q\), write its indices as \((i_q,j_q,k_q,r_q,s_q)\), put
\(y^{(q)}=\theta_qx^{(i_q)}+(1-\theta_q)x^{(j_q)}\). To check every subset
inequality for this relation, record the difference between its right and left sides:
\begin{equation}\label{eq:five-subset-slack}
 h_{q,J}=1-\sum_{\ell\in J}y^{(q)}_\ell
          -\sum_{\ell\notin J}x^{(k_q)}_{\ell},\qquad J\subseteq\{1,2,3\}.
\end{equation}
Thus \(h_{q,J}\ge0\) means that the corresponding subset inequality holds;
the selected contacts require equality. Set
\(g_{q,0}=-h_{q,\{r_q\}}\),
\(g_{q,1}=-h_{q,\{r_q,s_q\}}\), and
\(S(z)=5-\sum_{a=1}^{12}z_a\).
The first-order equations for minimizing \(S\) with these fourteen contacts
express its differential as a linear combination of the contact
differentials. Introduce their multipliers \(\alpha_{q,e}\), in
lexicographic order, and impose
\begin{equation}\label{eq:five-stationary-system}
\begin{split}
 g_{q,e}&=0\quad(1\le q\le7,\ e=0,1),\\
 \nabla_{z,\theta}\left(S+\sum_{q=1}^7\sum_{e=0}^1
                              \alpha_{q,e}g_{q,e}\right)&=0.
\end{split}
\end{equation}
These are \(33\) integer polynomial equations of degree at most two in
\(\zeta=(z,\theta,\alpha)\). Let \(\mathcal G_5\) be their left-hand
sides, with the fourteen contact equations first. The box is
\begin{equation}\label{eq:five-root-box}
 \mathcal B_5=c_5+[-10^{-18},10^{-18}]^{33},
\end{equation}
where the entries of \(10^{18}c_5\), read across each row and then
from top to bottom, are the following integers:
\begin{center}
\begin{tabular}{rrr}
\toprule
386014849027220836 & 308668158946037411 & 268568902039564312\\
386876036567496673 & 325869172515483314 & 285152485513455788\\
273710090048793452 & 423366639846964273 & 286361546133516014\\
390375173851731673 & 290721097484648125 & 426619246960605093\\
8963524113764484 & 7624145269613617 & 695747592913518932\\
292512163644851711 & 729080021651239933 & 151496450497535646\\
150319137153536034 & 296726341292503776 & 101711508212025312\\
282270889568991760 & 86031930355918667 & 390689774720592742\\
171907921572006435 & 403048362230368236 & 175690831925223518\\
501167820443288557 & 190637395276074643 & 605133557767341222\\
103834779310574281 & 634505174277405435 & 108947111983202378\\
\bottomrule
\end{tabular}
\end{center}
The variable order is \((z,\theta,\alpha)\) as specified above.

The matrix \(C_5=J_{\mathcal G_5}(c_5)^{-1}\) exists over \(\mathbb Q\).
With \(N_5(\zeta)=\zeta-C_5\mathcal G_5(\zeta)\), the second row of
\eqref{eq:exact-verification-summary} isolates a unique algebraic root.
On the box, all twelve free coordinates are positive,
\(0<\theta_q<1\), and \(0.0177<t_i<0.3234\). The fourteen defining
slacks vanish at the root; for the remaining forty-two,
\begin{equation}\label{eq:five-inactive-slacks}
 h_{q,J}>0.0177.
\end{equation}
The seven segment relations follow from Lemma~\ref{lem:positive-part}. The fixed
zeros give the smallest container \(\Sigma\). The total scale lies in
\((0.947696601064483024,0.947696601064483048)\), so
\begin{equation}\label{eq:five-value-isolation}
\begin{split}
 1.0551900248209902896764181722
 &<\lambda(\mathcal F_5)=1/S(z_*)\\
 &<1.0551900248209903163986418959.
\end{split}
\end{equation}

\subsection{Six members}\label{app:six}

The auxiliary generators and equivariant vertex labels are those of
Subsection~\ref{subsec:six}. To recover the geometry from the segment
contacts, use their three weights \(p,v,w\) as parameters. Put
\(d(p,v)=p(1+2v)-2\); solving five scalar contact equations gives
\begin{equation}\label{eq:six-recovery}
\begin{aligned}
 u&=\frac1{1-v},&a&=\frac{1-pv}{d(p,v)},\\
 b&=\frac{av}{(1-v)w}-1-a,&c&=1+a+\frac{a(1+2a)}{1+a+b},\\
 X&=a(u-1),&Y&=1+a-u(1+2a),\\
 g&=c(u-1).
\end{aligned}
\end{equation}
Direct substitution verifies
\begin{equation}\label{eq:six-basic-identities}
 r_5=wq_2+(1-w)r_1,\qquad
 p_1=(1-p)p_6+pr_4,\qquad r_2=(1-v)p_6+vr_1,
\end{equation}
except possibly the second coordinate of the middle equality. Its residual is
\begin{equation}\label{eq:six-middle-residual}
 [p_1-(1-p)p_6-pr_4]_2=\frac{E(p,v,w)}{w(v-1)d(p,v)},
\end{equation}
where
\begin{equation}\label{eq:six-E}
 E=p^2(1-v)w^2-\{p^2+v-pv(1+v)\}w+v(1-pv).
\end{equation}
The covering function becomes
\begin{equation}\label{eq:six-L}
 L=\frac{1-3Y}{3(b+g)}
   =\frac{w(pv^2+pv+p-v+1)}{3D_L},
\end{equation}
with
\[
 D_L=2pv^2w-pv^2-pvw^2+pvw+pw^2-pw-2vw+v+w.
\]
To select the best candidate found by the search, impose the first-order
conditions for \(L\) under the remaining contact \(E=0\). Where
\(E_w\ne0\), solving for \(w\) makes the two constrained derivatives
proportional to \(E_wL_p-E_pL_w\) and \(E_wL_v-E_vL_w\).
Subscripts here denote partial derivatives. Clearing their denominators
defines the integer polynomials
\begin{equation}\label{eq:six-H}
 H_p=3D_L^2(E_wL_p-E_pL_w),\qquad
 H_v=3D_L^2(E_wL_v-E_vL_w).
\end{equation}
The denominators cancel, and the coefficients of each polynomial have
greatest common divisor one. The exact candidate is selected by
\begin{equation}\label{eq:six-system}
 \mathcal H=(E,H_p,H_v)=0
\end{equation}
in the rational box
\begin{equation}\label{eq:six-box}
\begin{gathered}
 c_6=10^{-15}(990891271422741,700107453290074,847026682317516),\\
 \mathcal B_6=c_6+[-10^{-12},10^{-12}]^3.
\end{gathered}
\end{equation}

The matrix \(M_6=J_{\mathcal H}(c_6)^{-1}\) exists over \(\mathbb Q\).
The correction \(N_6(z)=z-M_6\mathcal H(z)\) satisfies the third row of
\eqref{eq:exact-verification-summary}. At its unique root
\((p_*,v_*,w_*)\), interval evaluation gives
\begin{equation}\label{eq:six-signs}
\begin{gathered}
 \min\{p_*,1-p_*,v_*,1-v_*,w_*,1-w_*,a,b,c,u-1,g,1-3Y\}>0,\\
 d(p_*,v_*)>0.378,\qquad D_L(p_*,v_*,w_*)>0.659,
 \qquad E_w(p_*,v_*,w_*)<-0.0037.
\end{gathered}
\end{equation}
All divisions are therefore valid. Since \(E=0\), the three vector
identities hold exactly and give the nine segment relations after rotation.
For containment, the same interval evaluation proves
\[
 \max\{Y+b,Y+a,Y-X,X+2Y+g-1\}<-6.
\]
Together with the positive scales, these inequalities put \(T_1,T_4\)
in \(\{(x,y):x\ge Y,\ y\ge Y,\ x+y\le1-Y\}\). Invariance under
\(R\) gives containment of the other four triangles. The supporting-edge
argument in Subsection~\ref{subsec:six} proves minimality, and
\begin{equation}\label{eq:six-value-isolation}
 1.05720611358<\lambda(\mathcal F_6)=L(p_*,v_*,w_*)<1.05720611360.
\end{equation}
This completes the proof of Proposition~\ref{prop:exact-verification}.

\section{Coordinates of the rational 303-member family}\label{app:large-data}

Tables~\ref{tab:large-coordinates} and~\ref{tab:large-coordinates-last} define the family in
Theorem~\ref{thm:main} by
\[
 \mathbf p^{(k)}=10^{-9}(X_k,Y_k,Z_k),\qquad
 P_{k,i}=T(\rho^i\mathbf p^{(k)})\quad(0\le k\le100,\ 0\le i\le2).
\]
The entries are integers; $k=0$ denotes the support and $k=1,\ldots,100$
the consecutive bridges. Section~\ref{sec:finite-bounds} verifies the segment
relations and total scale directly from these coordinates.

\par\medskip
\noindent\begin{minipage}{\linewidth}
\makeatletter\def\@captype{table}\makeatother
\centering
\caption{Integer coordinate numerators for $0\le k\le59$.}
\label{tab:large-coordinates}
{\small
\setlength{\tabcolsep}{4pt}
\begin{tabular}{rrrr@{\hspace{1.2em}}rrrr}
\toprule
$k$ & $X_k$ & $Y_k$ & $Z_k$ & $k$ & $X_k$ & $Y_k$ & $Z_k$\\
\midrule
0 & 310226993 & 415666856 & 0 & 30 & 300895260 & 415040422 & 283725044\\
1 & 309931259 & 415665639 & 274107058 & 31 & 300566393 & 414996144 & 284096350\\
2 & 309634609 & 415663193 & 274404925 & 32 & 300236245 & 414950168 & 284470686\\
3 & 309337064 & 415659507 & 274704942 & 33 & 299904877 & 414902483 & 284848006\\
4 & 309038322 & 415654565 & 275007105 & 34 & 299572298 & 414853072 & 285228264\\
5 & 308738731 & 415648359 & 275311725 & 35 & 299238441 & 414801907 & 285611476\\
6 & 308438119 & 415640875 & 275618467 & 36 & 298903359 & 414748977 & 285997724\\
7 & 308136404 & 415632098 & 275927517 & 37 & 298566952 & 414694248 & 286386983\\
8 & 307833714 & 415622016 & 276238970 & 38 & 298229152 & 414637689 & 286779387\\
9 & 307529893 & 415610613 & 276552713 & 39 & 297890099 & 414579301 & 287175008\\
10 & 307225037 & 415597879 & 276868913 & 40 & 297549630 & 414519037 & 287573743\\
11 & 306919187 & 415583800 & 277187488 & 41 & 297207855 & 414456895 & 287975765\\
12 & 306612220 & 415568359 & 277508412 & 42 & 296864632 & 414392829 & 288381003\\
13 & 306304161 & 415551543 & 277831823 & 43 & 296520032 & 414326828 & 288789616\\
14 & 305994990 & 415533336 & 278157711 & 44 & 296174019 & 414258865 & 289201557\\
15 & 305684800 & 415513728 & 278486105 & 45 & 295826681 & 414188935 & 289616873\\
16 & 305373450 & 415492698 & 278816933 & 46 & 295477894 & 414116990 & 290035511\\
17 & 305061049 & 415470238 & 279150348 & 47 & 295127689 & 414043015 & 290457619\\
18 & 304747453 & 415446323 & 279486260 & 48 & 294775949 & 413966960 & 290883208\\
19 & 304432818 & 415420948 & 279824822 & 49 & 294422703 & 413888806 & 291312420\\
20 & 304117089 & 415394096 & 280165894 & 50 & 294067929 & 413808525 & 291745253\\
21 & 303800181 & 415365745 & 280509552 & 51 & 293711763 & 413726123 & 292181733\\
22 & 303482092 & 415335876 & 280855897 & 52 & 293353970 & 413641520 & 292621776\\
23 & 303162877 & 415304481 & 281204946 & 53 & 292994631 & 413554708 & 293065650\\
24 & 302842511 & 415271540 & 281556658 & 54 & 292633693 & 413465650 & 293513301\\
25 & 302520965 & 415237035 & 281911078 & 55 & 292271291 & 413374349 & 293964788\\
26 & 302198162 & 415200940 & 282268258 & 56 & 291907259 & 413280741 & 294420024\\
27 & 301874282 & 415163258 & 282628280 & 57 & 291541607 & 413184797 & 294879211\\
28 & 301549112 & 415123949 & 282990993 & 58 & 291174294 & 413086479 & 295342376\\
29 & 301222790 & 415083012 & 283356621 & 59 & 290805430 & 412985788 & 295809571\\
\bottomrule
\end{tabular}
}
\end{minipage}
\par

\par\medskip
\noindent\begin{minipage}{\linewidth}
\makeatletter\def\@captype{table}\makeatother
\centering
\caption{Integer coordinate numerators for $60\le k\le100$.}
\label{tab:large-coordinates-last}
{\small
\setlength{\tabcolsep}{4pt}
\begin{tabular}{rrrr@{\hspace{1.2em}}rrrr}
\toprule
$k$ & $X_k$ & $Y_k$ & $Z_k$ & $k$ & $X_k$ & $Y_k$ & $Z_k$\\
\midrule
60 & 290434672 & 412882600 & 296280770 & 81 & 282208599 & 410075422 & 307249241\\
61 & 290062281 & 412776957 & 296756327 & 82 & 281792884 & 409906842 & 307829693\\
62 & 289688092 & 412668780 & 297236020 & 83 & 281374671 & 409734611 & 308416254\\
63 & 289312200 & 412558066 & 297720042 & 84 & 280953837 & 409558621 & 309009026\\
64 & 288934451 & 412444736 & 298208354 & 85 & 280530386 & 409378816 & 309608204\\
65 & 288555003 & 412328804 & 298701121 & 86 & 280104264 & 409195115 & 310213857\\
66 & 288173571 & 412210150 & 299198266 & 87 & 279675415 & 409007434 & 310826124\\
67 & 287790326 & 412088792 & 299700101 & 88 & 279243643 & 408815622 & 311445186\\
68 & 287405201 & 411964672 & 300206484 & 89 & 278809001 & 408619636 & 312071318\\
69 & 287018148 & 411837740 & 300717539 & 90 & 278371379 & 408419360 & 312704559\\
70 & 286629019 & 411707909 & 301233386 & 91 & 277930726 & 408214700 & 313345120\\
71 & 286237957 & 411575186 & 301754196 & 92 & 277486967 & 408005548 & 313993154\\
72 & 285844820 & 411439486 & 302279882 & 93 & 277039911 & 407791738 & 314648882\\
73 & 285449594 & 411300762 & 302810645 & 94 & 276589564 & 407573194 & 315312601\\
74 & 285052298 & 411158979 & 303346537 & 95 & 276135737 & 407349741 & 315984445\\
75 & 284652770 & 411014037 & 303887619 & 96 & 275678413 & 407121284 & 316664724\\
76 & 284251077 & 410865915 & 304434098 & 97 & 275217393 & 406887636 & 317353610\\
77 & 283847305 & 410714601 & 304985923 & 98 & 274752737 & 406648732 & 318051409\\
78 & 283441101 & 410559917 & 305543131 & 99 & 274284071 & 406404284 & 318758270\\
79 & 283032682 & 410401898 & 306106108 & 100 & 273811551 & 406154271 & 319474680\\
80 & 282621828 & 410240409 & 306674717 &  &  &  & \\
\bottomrule
\end{tabular}
}
\end{minipage}
\par

\section*{Acknowledgement of AI assistance}
OpenAI's GPT models assisted with the exposition, proof development and
refinement, and exact computation and numerical verification.
The authors take responsibility for the mathematical content and final manuscript.
\bibliographystyle{amsplain}
\bibliography{references}
\end{document}